\documentclass{amsart}

\usepackage[stable]{footmisc}
\usepackage{tikz-cd,amssymb,amsmath,amsthm,mathscinet,mathrsfs,anyfontsize,latexsym,graphics,mathdots,hyperref,bm}
\usepackage[all]{xy}

\newcommand{\grdim}{\operatorname{grdim}}

\newcommand{\rk}{\operatorname{rk}}
\newcommand{\inter}{R}
\newcommand{\iprec}{\prec\prec}
\newcommand{\grph}{G}

\newcommand{\Tors}{\mathfrak{T}}
\newcommand{\Torf}{\mathfrak{F}}
\newcommand{\Cls}{\mathfrak{C}}
\newcommand{\BZ}{\mathcal{R}}
\newcommand{\linT}{\mathcal{T}}
\newcommand{\bss}{\mathtt{b}}
\newcommand{\altbss}{\mathtt{c}}
\newcommand{\bssaux}{\mathtt{f}}
\newcommand{\bssauxx}{\mathtt{g}}
\newcommand{\mtch}{\mathcal{M}}
\newcommand{\mtching}{\mathfrak{R}}
\newcommand{\Mult}{\mathfrak{M}}
\renewcommand{\subset}{\subseteq}
\renewcommand{\supset}{\supseteq}
\newcommand{\JH}{\operatorname{JH}}
\newcommand{\cn}{\mu}
\newcommand{\SI}{SI}
\newcommand{\prm}{\lambda}
\newcommand{\std}{\zeta}
\newcommand{\op}{\circ}
\newcommand{\GL}{\operatorname{GL}}
\newcommand{\Irrcomp}{\operatorname{Comp}}
\newcommand{\one}{\mathbf{1}}
\newcommand{\bfd}{\bm{d}}

\newcommand{\m}{\mathfrak{m}}
\newcommand{\n}{\mathfrak{n}}
\newcommand{\bs}{\backslash}
\newcommand{\C}{\mathbb{C}}
\newcommand{\Z}{\mathbb{Z}}
\newcommand{\rest}{\big |}
\newcommand{\N}{\mathbb{N}}
\newcommand{\Hom}{\operatorname{Hom}}
\newcommand{\Ext}{\operatorname{Ext}}
\newcommand{\ext}{\operatorname{ext}}
\newcommand{\Irr}{\operatorname{Irr}}
\newcommand{\soc}{\operatorname{soc}}                          
\newcommand{\tr}{\operatorname{tr}}
\newcommand{\fl}{\mathcal{F}}
\newcommand{\rshft}[1]{\overset{\rightarrow}{#1}}               
\newcommand{\id}{\operatorname{id}}
\newcommand{\End}{\operatorname{End}}
\newcommand{\Ker}{\operatorname{Ker}}
\newcommand{\Coker}{\operatorname{Coker}}
\newcommand{\abs}[1]{\left|{#1}\right|}

\newcommand{\KJ}{\mathcal{K}}
\newcommand{\sprt}{\mathrel{\reflectbox{\rotatebox[origin=c]{315}{$\pitchfork$}}}}
\newcommand{\indx}{\operatorname{ind}}
\newcommand{\hind}{\operatorname{h-ind}}
\newcommand{\Irrsat}{\Irr_{\sigma\text{-sat}}}
\newcommand{\Irrsatb}{\Irr_{\sigma\text{-basic}}}
\newcommand{\Reps}{\mathcal{C}}
\newcommand{\Repsred}{\Reps_{\sigma\text{-red}}}
\newcommand{\Irred}{\Irr_{\sigma\text{-red}}}

\makeatletter
\newenvironment{subtheorem}[1]{%
  \def\subtheoremcounter{#1}%
  \refstepcounter{#1}%
  \protected@edef\theparentnumber{\csname the#1\endcsname}%
  \setcounter{parentnumber}{\value{#1}}%
  \setcounter{#1}{0}%
  \expandafter\def\csname the#1\endcsname{\theparentnumber.\Alph{#1}}%
  \expandafter\def\csname theH#1\endcsname{thm.\theparentnumber.\Alph{#1}}%
  \unskip\ignorespaces
}{%
  \setcounter{\subtheoremcounter}{\value{parentnumber}}%
  \ignorespacesafterend
}
\makeatother
\newcounter{parentnumber}

\newtheorem{theorem}{Theorem}[section]
\newtheorem{lemma}[theorem]{Lemma}
\newtheorem{definition}[theorem]{Definition}
\newtheorem{proposition}[theorem]{Proposition}
\newtheorem{corollary}[theorem]{Corollary}
\newtheorem{conjecture}[theorem]{Conjecture}
\theoremstyle{remark}
\newtheorem{remark}[theorem]{Remark}
\newtheorem{example}[theorem]{Example}
\newtheorem{question}[theorem]{Question}

\numberwithin{equation}{section}

\begin{document}

\title[Irreducible components of nilpotent varieties {II}]{A binary operation on irreducible components of
Lusztig's nilpotent varieties {II}: applications and conjectures for representations of $\GL_n$ over a non-archimedean local field}

\author{Erez Lapid}
\email{erez.m.lapid@gmail.com}
\address{Department of Mathematics\\Weizmann Institute of Science\\Rehovot 7610001 Israel}
\author{Alberto M\'{\i}nguez}
\email{minguez.alberto@gmail.com}
\address{Faculty of Mathematics\\University of Vienna\\Oskar-Morgenstern-Platz 1\\
1090 Wien Austria}

\date{\today}

\maketitle

\begin{abstract}
In the first part of the paper we defined and studied a binary operation on the set of irreducible components of Lusztig's nilpotent varieties of a quiver.
For type $A$ we conjecture, following Geiss and Schr\"oer, that this operation is compatible with taking the socle of parabolic induction of representations of general
linear groups over a local non-archimedean field, at least when one of the irreducible components is rigid.
We verify this conjecture in special cases.
\end{abstract}

\setcounter{tocdepth}{1}
\tableofcontents

\section{Introduction}

Let $F$ be a non-archimedean local field.
The representation theory of the locally compact, totally disconnected groups $\GL_n(F)$, $n\ge0$
plays a particularly important role in the general theory of representations of $p$-adic groups.
(All representations considered here are tacitly complex and smooth.)
It was studied by Bernstein and Zelevinsky in their fundamental work in the 1970s, where they highlight the advantage of
working with all $n$'s together \cite{MR0425030, MR579172, MR584084}.
(For simplicity, we refer to it collectively as representation theory of $\GL$.)
Normalized parabolic induction, which following the usual convention is written as $\pi_1\times\pi_2$,
endows the category of finite length representations of $\GL_n(F)$, $n\ge0$ with the structure
of a ring category over $\C$ (i.e., a locally finite $\C$-linear abelian monoidal category, where the tensor functor is bilinear and  biexact, and $\End(\one)=\C$).
The associativity constraints are given by transitivity of parabolic induction and the unit object is given by the irreducible (one-dimensional)
representation of $\GL_0=1$.

Compared to a general reductive group, $\GL_n$ is simpler in that $L$-packets are singletons and roughly speaking,
``all supercuspidal representations look the same''.

It is well known that the representation theory of $\GL$ is closely related to the representation theory of quantum groups of type $A$ \cite{MR1985725}.
In particular, via Lusztig's canonical basis, the irreducible objects are classified by the set $\Irrcomp$ of
irreducible components of nilpotent varieties of type $A$, which are in turn indexed by multisegments \cite{MR1035415, MR1088333}.
We denote this correspondence by $C\mapsto\pi(C)$.
The nilpotent varieties classify modules of a given graded dimension of the preprojective algebra.
(For a general quiver, one adds a nilpotency condition.)

A special role is played by the $\square$-irreducible representations, namely those (irreducible) $\pi$'s
such that $\pi\times\pi$ is irreducible. For instance, every unitarizable irreducible representation
is $\square$-irreducible. The $\square$-irreducible representations and their analogs (``real'' in Leclerc's terminology) were studied by many people
including Geiss, Hernandez, Kang, Kashiwara, Kim, Leclerc, Oh, Schr\"oer and others
in the context of monoidal categorification of cluster algebras \cite{MR2682185, MR2242628, MR3758148}.
We refer the reader to \cite{MR3966729} and \cite{1902.01432} for recent surveys on this topic.
It is conjectured that $\pi(C)$ is $\square$-irreducible if and only if $C$ is rigid, i.e.,
$C$ contains an open orbit.
This was proved in the so-called regular case in \cite{MR3866895}, together with
a purely combinatorial characterization of $\square$-irreducibility in this case.

One of the nice properties of $\square$-irreducible representations is that if $\pi_1$ or $\pi_2$ is $\square$-irreducible,
then the socle $\soc(\pi_1\times\pi_2)$
is irreducible and occurs with multiplicity one in the Jordan--H\"older sequence of $\pi_1\times\pi_2$.
This raises the question of describing this socle explicitly. This is our motivating question.
It naturally leads to look for a similar operation on $\Irrcomp$.
Such an operation was conceived in the first part of the paper \cite{2103.12027} (for any quiver).
Namely, given any two irreducible components $C_1,C_2\in\Irrcomp$, a third one, denoted
$C_1*C_2$, was defined to be the closure of the constructible set formed by all
possible extensions of $x_1$ by $x_2$ for all $(x_1,x_2)$ in a suitable open subset of $C_1\times C_2$.

It is natural to expect that $\pi(C_1*C_2)$ is always a subquotient,
or perhaps even a subrepresentation, of $\pi(C_1)\times\pi(C_2)$. We formulate this as our main conjecture
(Conjecture \ref{conj: subconj}).
It is closely related to a conjecture made by Geiss and Schr\"oer in \cite{MR2115084}
in the context of dual canonical bases of simply laced Dynkin diagrams.
A special case of Conjecture \ref{conj: subconj} was proposed in \cite{MR4169050}.
(We were unaware of \cite{MR2115084} at the time.)
In the case where $C_1$ or $C_2$ is rigid, one gets an especially nice necessary and sufficient (conjectural)
irreducibility criterion for $\pi(C_1)\times\pi(C_2)$, namely $C_1*C_2=C_2*C_1$.
As our best evidence so far for the main conjecture,
we show that it holds if $C_1$ or $C_2$ is regular and rigid (namely, the case analyzed in \cite{MR3866895}).
This strengthens the main result of \cite{MR4169050} which proves the weaker conjecture of [ibid.]
in a more restricted case.

We now describe the contents of the paper in more detail.

We first review and explicate the pertinent material of \cite{2103.12027} for $Q=A_n$ (\S\ref{sec: An}).
The computations of $C_1*C_2$, $\hom(C_1,C_2)$ and $\ext^1(C_1,C_2)$ (the generic dimensions
of $\Hom$ and $\Ext^1$) become concrete and efficient, as long as we are content
with randomized algorithms.
We give examples of rigid and non-rigid irreducible components.

In \S\ref{sec: GLn} we recall the bijection between $\Irrcomp$ and (a representative class of) irreducible representations of $\GL$.
We also discuss a variant and special cases of the conjectures of Geiss and Schr\"oer.

The notion of $\square$-irreducible representations and their basic properties, following
Kang--Kashiwara--Kim--Oh \cite{MR3314831, MR3758148}, are discussed in \S\ref{sec: sqr}.

The main new conjecture, relating the binary operation $*$ to the socle of the parabolic induction,
at least in the case where one of the irreducible components is rigid, is stated in \S\ref{sec: mainconj}.
We give a few sanity checks and prove it in the special case where $C_1$ or $C_2$ corresponds to a single segment.

The results of \cite{MR3866895}, characterizing $\square$-irreducibility in the regular case,
are reviewed in \S\ref{sec: regms}.
We then state our best evidence so far for Conjecture \ref{conj: subconj}, to wit, that it holds if
$C_1$ or $C_2$ is regular and rigid (Theorem \ref{thm: cor}).
As an aside, we give a simple necessary and sufficient condition for a regular representation to be a prime element in the Bernstein--Zelevinsky ring.
We also formulate an intriguing, purely geometric, duality conjecture of Anton Mellit for
Schubert varieties of type $A$.

In \S\ref{sec: Basic} we consider a special case of the
Baumann--Kamnitzer--Tingley decompositions of $\Irrcomp$ \cite{MR3270589} and its analogue to representation theory of $\GL$.
This is used in the proof of Theorem \ref{thm: cor} in \S\ref{sec: endofproof}, which also builds on ideas from \cite{MR3573961, MR3866895, MR4169050}.
Clearly, more ideas are needed to go beyond the regular case.

We end the paper with a highly speculative conjectural characterization of the irreducible subquotients of any product
$\pi_1\times\dots\times\pi_k$ (\S\ref{sec: odds}).

\subsection*{Acknowledgement}
We are grateful to Bernard Leclerc for useful correspondence and to Max Gurevich for helpful discussions.
The first-named author would like to thank the Hausdorff Research Institute for Mathematics in Bonn for its
generous hospitality in July 2021.
Conversations with Jan Schr\"oer during that time were particularly helpful and we are very grateful to him.
We also thank Anton Mellit for useful discussions and for allowing us to include Conjecture \ref{conj: Mellit} here.
Special thanks are due to Avraham Aizenbud for many conversations on the subject.
Last but not least, we are greatly indebted to the referee for reading the paper carefully and making many useful suggestions.

\section{Type \texorpdfstring{$A_n$}{An}} \label{sec: An}

We consider the quiver $Q$ of type $A_n$ with the standard orientation
\[
\bullet\rightarrow\bullet\rightarrow\dots\rightarrow\bullet
\]
Thus, a representation of $Q$ is a graded vector space $V=\oplus_{i=1}^nV_i$, together with a degree $1$ linear transformation
$T_+:V\rightarrow V$.\footnote{As in \cite{2103.12027} all vector spaces and varieties in this section are over a fixed algebraically closed field.}
In particular, $T_+$ is nilpotent of degree at most $n$.
We will freely use the terminology and the results of \cite{2103.12027}.
(However, we no longer use $I$ to denote the index set of the vertices.)

We identify the set $\Psi$ of positive roots with \emph{segments} $[i,j]$, $1\le i\le j\le n$.
(If $\alpha_1,\dots,\alpha_n$ are the simple roots in the usual ordering, then
$[i,j]$ corresponds to $\alpha_i+\dots+\alpha_j$.)

\subsection{} \label{sec: MSAn}
Recall that $\Psi$ classifies the indecomposable modules.
In the case at hand, to each segment $\Delta=[i,j]$ the corresponding representation $M_Q(\Delta)$ is a Jordan block whose
graded dimension $\grdim\Delta$ is the indicator function of $[i,j]$ (with $T_+^{j-i}\ne0$).
The simple objects (corresponding to the simple roots) are $M_Q([i,i])$, $i=1,\dots,n$.
The projective ones are $M_Q([i,n])$, $i=1,\dots,n$.

The abelian monoid $\Mult=\Mult^n=\N\Psi$ freely generated by $\Psi$ classifies all $Q$-representations up to isomorphism.
The elements of $\Mult$ are called multisegments. A typical element $\m\in\Mult$ is written as
\[
\m=\Delta_1+\dots+\Delta_k
\]
where $\Delta_1,\dots,\Delta_k$ are segments. (Their order does not matter.)
We write $M_Q(\m)$ for the corresponding representation. Extending $\grdim$ additively from $\Psi$ to $\Mult$, we have a disjoint union decomposition
\[
\Mult=\coprod_{\bfd\in\N^n}\Mult(\bfd)\text{ where }\Mult(\bfd)=\{\m\in\Mult\mid\grdim\m=\bfd\}.
\]
Fix $V=\oplus_{i=1}^nV_i$ of graded dimension $\bfd=(\bfd(1),\dots,\bfd(n))\in\N^n$
and set $V_i=0$ and $\bfd(i)=0$ if $i\notin\{1,\dots,n\}$.
Let $G_V=\prod_{i=1}^n\GL(V_i)$.
The $G_V$-orbits on
\[
R_Q(V)=\{T_+:V\rightarrow V\mid T_+(V_i)\subset V_{i+1}\text{ for all }i\}
\]
under conjugation are parameterized by $\Mult(\bfd)$.
It is easy to determine the multisegment corresponding to the $G_V$-orbit of $T_+$
from the ranks of $T_+^j\rest_{V_i}$, $1\le i\le i+j\le n$ (see e.g., \cite[\S1]{MR1371654}).

Let $\Delta=[a,b]$ and $\Gamma=[c,d]$ be two segments.
Following Zelevinsky, we say that $\Delta$ precedes $\Gamma$ if $a+1\le c\le b+1\le d$ and denote it by $\Delta\prec\Gamma$.
We have
\[
\dim\Hom_Q(M_Q(\Gamma),M_Q(\Delta))=\begin{cases}1&\text{if }\Delta\prec\rshft\Gamma,\\0&\text{otherwise,}\end{cases}
\]
where $\rshft\Gamma=[c+1,d+1]$, and
\[
\dim\Ext^1_Q(M_Q(\Delta),M_Q(\Gamma))=\begin{cases}1&\text{if }\Delta\prec\Gamma,\\0&\text{otherwise.}\end{cases}
\]

Thus, for any multisegments $\m,\n$ we have
\begin{subequations}
\begin{equation} \label{eq: notprec}
\begin{gathered}
\Ext_Q^1(M_Q(\m),M_Q(\n))=0\iff\\\Delta\not\prec\Gamma\text{ for every segment $\Delta$ of $\m$ and $\Gamma$ of $\n$}
\end{gathered}
\end{equation}
and
\begin{equation}\label{eq: notprec2}
\begin{gathered}
\Hom_Q(M_Q(\n),M_Q(\m))=0\iff\\\Delta\not\prec\rshft\Gamma\text{ for every segment $\Delta$ of $\m$ and $\Gamma$ of $\n$.}
\end{gathered}
\end{equation}
\end{subequations}

For any graded vector space $V=\oplus_{i=1}^nV_i$, the open $G_V$-orbit in $R_Q(V)$ corresponds to the multisegment $\m\in\Mult(\bfd)$ such that
$\Delta_i\not\prec\Delta_j$ for any two segments in $\m$.

Recall that we can identify
\[
R_{Q^{\op}}(V)=\{T_-:V\rightarrow V\mid T_-(V_i)\subset V_{i-1}\text{ for all }i\}
\]
with the dual space of $R_Q(V)$ by the pairing $(T_+,T_-)=\tr T_+T_-=\tr T_-T_+$.
It follows from \cite[(2.11)]{2103.12027} that, for any two segments $\Delta,\Gamma$ we have
\begin{gather*}
\dim\Hom_{Q^{\op}}(M_{Q^{\op}}(\Delta),M_{Q^{\op}}(\Gamma))=\begin{cases}1&\text{if }\Delta\prec\rshft\Gamma,\\0&\text{otherwise,}\end{cases}\\
\dim\Ext^1_{Q^{\op}}(M_{Q^{\op}}(\Gamma),M_{Q^{\op}}(\Delta))=\begin{cases}1&\text{if }\Delta\prec\Gamma,\\0&\text{otherwise.}\end{cases}
\end{gather*}
Hence, for every two multisegments $\m$ and $\n$ we have
\begin{gather*}
\Ext_{Q^{\op}}^1(M_{Q^{\op}}(\n),M_{Q^{\op}}(\m))=0\iff\\
\Delta\not\prec\Gamma\text{ for every segment $\Delta$ of $\m$ and $\Gamma$ of $\n$}
\end{gather*}
and
\begin{equation}\label{eq: notprec'}
\begin{gathered}
\Hom_{Q^{\op}}(M_{Q^{\op}}(\m),M_{Q^{\op}}(\n))=0\iff\\
\Delta\not\prec\rshft\Gamma\text{ for every segment $\Delta$ of $\m$ and $\Gamma$ of $\n$.}
\end{gathered}
\end{equation}

A special feature of $Q$ is that it is isomorphic to $Q^{\op}$ via $i\mapsto n+1-i$.
This yields an (equivariant) equivalence of the categories of representations of $Q$ and $Q^{\op}$.
Composing it with the standard duality $M\mapsto M^*$ between finite-dimensional representations of $Q$ and $Q^{\op}$
we get a self-duality on the category of finite-dimensional representations of $Q$
(or of $Q^{\op}$), which we denote by ${}^\vee$.
For any segment $\Delta=[a,b]$ we have
\[
M_Q(\Delta)^\vee=M_Q(\Delta^\vee),\ \  M_{Q^{\op}}(\Delta)^\vee=M_{Q^{\op}}(\Delta^\vee)
\]
where $\Delta^\vee=[n+1-b,n+1-a]$. Extending this involution to the set of multisegments by additivity we have
\[
M_Q(\m)^\vee=M_Q(\m^\vee),\ \ M_{Q^{\op}}(\m)^\vee=M_{Q^{\op}}(\m^\vee).
\]

\subsection{} \label{sec: basicPP}
Recall that $\overline{Q}$ is the double quiver (i.e., all orientations simultaneously)
\[
\bullet\leftrightarrow\bullet\leftrightarrow\dots\leftrightarrow\bullet
\]
and the preprojective algebra $\Pi$ is the finite-dimensional quotient of the path algebra of $\overline{Q}$ by the relations
\[
e_{i+1,i}e_{i,i+1}-e_{i-1,i}e_{i,i-1},\ \ i=1,\dots,n.
\]
Here $e_{r,s}$, $s=r\pm1$ is the arrow from $r$ to $s$, interpreted as $0$ unless $1\le r,s\le n$.

The nilpotent variety $\Lambda(V)=R_\Pi(V)$ is given by
\[
\Lambda(V)=\{(T_+,T_-)\in R_Q(V)\times R_{Q^{\op}}(V)\mid T_+T_-=T_-T_+\}
\]
and the maps
\[
\pi_Q:\Lambda(V)\rightarrow R_Q(V),\ \ \pi_{Q^{\op}}:\Lambda(V)\rightarrow R_{Q^{\op}}(V)
\]
are the canonical projections.

Recall that we denote the set of irreducible components of $\Lambda(V)$ by
$\Irrcomp(V)$, or simply by $\Irrcomp(\bfd)$ if $\bfd=\grdim V$, and let
\[
\Irrcomp=\bigcup_{\bfd\in\N^n}\Irrcomp(\bfd).
\]
We have a bijection
\[
\prm_Q:\Mult(\bfd)\rightarrow\Irrcomp(\bfd),\ \m\mapsto\overline{\cn_Q(\m)}
\]
where $\cn_Q(\m)$ is the inverse image under $\pi_Q$ of the $G_V$-orbit of $M_Q(\m)$.
Similarly, we have a bijection
\[
\prm_{Q^{\op}}:\Mult(\bfd)\rightarrow\Irrcomp(\bfd).
\]
In fact,
\begin{equation} \label{eq: prm*}
\prm_Q(\m)^*=\prm_{Q^{\op}}(\m)\ \ \forall\m\in\Mult.
\end{equation}
We also have
\[
\prm_Q(\m)=\prm_{Q^{\op}}(\m^t)
\]
where $\m^t$ is the M\oe glin--Waldspurger involution \cite[Proposition II.6]{MR863522} -- see also \cite{MR1371654, MR1734897}.

We recall that by definition, a $\Pi$-module $x$ is \emph{rigid} if $\Ext^1_\Pi(x,x)=0$.
Also, an irreducible component $C\in\Irrcomp(V)$ is rigid, if it contains a rigid element, in which
case the rigid elements form an open orbit in $C$. See \cite[\S4]{2103.12027} for more details.

Observe that by \eqref{eq: notprec2} and \eqref{eq: notprec'}, for any two multisegments $\m,\n$
\begin{equation}\label{eq: notprecPP}
\begin{split}
\text{if $\Delta\not\prec\rshft\Gamma$ for every segment $\Delta$ of $\m$ and $\Gamma$ of $\n$, then}\\
\hom_\Pi(\prm_Q(\n),\prm_Q(\m))=\hom_\Pi(\prm_{Q^{\op}}(\m),\prm_{Q^{\op}}(\n))=0.
\end{split}
\end{equation}
(Recall that by definition,
\[
\hom_\Pi(C_1,C_2)=\min_{(x_1,x_2)\in C_1\times C_2}\dim\Hom_\Pi(x_1,x_2)
\]
for any irreducible components $C_1$, $C_2$. Similarly, we will also use the notation $\ext_\Pi^1(C_1,C_2)$ defined analogously.)

The bijection $\Lambda(V)\rightarrow\Lambda(V^*)$ takes $(T_+,T_-)$ to $(T_-^*,T_+^*)$.
On the other hand, the involution $i\mapsto n+1-i$ of $\overline{Q}$ gives rise to an involution of $\Pi$ and hence to an autoequivalence of the module category of $\Pi$.
Composing it with $M\mapsto M^*$ we get another self-duality $M\mapsto M^\vee$ (which commutes with $^*$).
At the level of irreducible components, we have
\[
\prm_Q(\m)^\vee=\prm_Q(\m^\vee),\ \ \prm_{Q^{\op}}(\m)^\vee=\prm_{Q^{\op}}(\m^\vee).
\]

Finally, we recall the binary operation $*$ on $\Irrcomp$ defined in \cite[\S 3]{2103.12027}.
Given $C_i\in\Irrcomp(V^i)$, $i=1,2$ and $V=V^1\oplus V^2$, $C_1*C_2$ is the Zariski closure of the
set of $x\in\Lambda(V)$ that fit in a short exact sequence
\[
0\rightarrow x_2\rightarrow x\rightarrow x_1\rightarrow0
\]
where $(x_1,x_2)\in C_1\times C_2$ and $\dim\Ext_\Pi^1(x_1,x_2)$ is minimal.

Recall that by definition, $C_1$ and $C_2$ \emph{strongly commute} if $\Ext_\Pi^1(x_1,x_2)=0$
for some $(x_1,x_2)\in C_1\times C_2$. Equivalently, the Zariski closure of the set of modules isomorphic
to $x_1\oplus x_2$ for some $(x_1,x_2)\in C_1\times C_2$ is an irreducible component (necessarily equal
to $C_1*C_2=C_2*C_1$) -- see \cite[Corollary 3.3]{2103.12027}.

If $C_1$ or $C_2$ is rigid, strong commutativity is equivalent to the a priori weaker condition
$C_1*C_2=C_2*C_1$ \cite[Proposition 6.1]{2103.12027}. In this case, we sometimes omit the adjective ``strong''.

By \cite[(3.2)]{2103.12027} we have
\begin{equation} \label{eq: cont*}
(C_1*C_2)^\vee=C_2^\vee*C_1^\vee
\end{equation}
for any $C_1,C_2\in\Irrcomp$.

\subsection{} \label{sec: explint}
Next, we expound \cite[\S9]{2103.12027} in the case at hand,
and in particular the maps $\linT_{M;N}:\Hom_Q(M,N)\rightarrow\Hom_Q(M,\tau N)$ of \cite[(9.8)]{2103.12027}
defined for any $\Pi$-modules $M$ and $N$. The functor $\tau$ (and its left adjoint $\tau^-$) are essentially the Coxeter functors
with respect to $Q$.
Concretely, $\tau$ (resp., $\tau^-$) takes a pair $(\oplus_{i=1}^nV_i,T_+)$ to the pair
$(\oplus_{i=1}^nV_i',T_+')$ where (letting $V_0=V_{n+1}=0$)
\[
V'_i=\Ker(T_+^{n-i+1}\rest_{V_{i-1}})\ (\text{resp., }\Coker(T_+^i\rest_{V_1}))\ \ i=1,\dots,n
\]
and $T'_+$ is the map induced by $T_+$.
In particular, $\tau M_Q(\Delta)=M_Q(\rshft\Delta)$ where by convention $M_Q([i,n+1])=0$ for all $i$,
while $\tau^-(M_Q([a,b]))=M_Q([a-1,b-1])$ again, interpreted as $0$ if $a=1$.

Let $M\in\cn_Q(\m)$, for a multisegment
\[
\m=\sum_{i\in I}\Delta_i.
\]
Then, there exists a graded basis $\bssaux_{i,r}$, $i\in I$, $r\in\Delta_i$ (of degree $r$) for $M$
(as a graded vector space) such that
\[
e_{r,r+1}\bssaux_{i,r}=\begin{cases}\bssaux_{i,r+1}&\text{if }r+1\in\Delta_i,\\0&\text{otherwise.}\end{cases}
\]
Let
\[
U_{\m}=\{(i,j)\in I\times I\mid\Delta_j\prec\Delta_i\},
\]
so that $\# U_{\m}=\dim\Hom_Q(M_Q(\m),\tau(M_Q(\m)))$.
By \cite[Lemme II.4]{MR863522} there exist scalars $x_{i,j}$, $(i,j)\in U_{\m}$ such that
\[
e_{r+1,r}\bssaux_{i,r+1}=\sum_{j\in I\mid r\in\Delta_j\text{ and }(i,j)\in U_{\m}}x_{i,j}\bssaux_{j,r}
\]
for all $i\in I$ and $r$ such that $r+1\in\Delta_i$.
We call $x_{i,j}$ the coordinates of $M$. \label{sec: coor}
They encode the data of $M$ as a $\Pi$-module (cf.\ \cite[\S9]{2103.12027}).
Of course, the coordinates depend on the choice of the basis $\bssaux_{i,r}$.

Suppose that in addition $N\in\cn_Q(\n)$ where
\[
\n=\sum_{j\in J}\Gamma_j.
\]
Denote by $\bssauxx_{j,r}$, $j\in J$, $r\in\Gamma_j$ a graded basis for $N$ such that
\[
e_{r,r+1}\bssauxx_{j,r}=\begin{cases}\bssauxx_{j,r+1}&\text{if }r+1\in\Gamma_j,\\0&\text{otherwise.}\end{cases}
\]
Let $y_{i,j}$, $(i,j)\in U_{\n}$ be the coordinates of $N$.
Let
\begin{equation} \label{eq: UandV}
U_{\m;\n}=\{(i,j)\in I\times J\mid\Gamma_j\prec\Delta_i\},\ \
V_{\m;\n}=\{(i,j)\in I\times J\mid\Gamma_j\prec\rshft\Delta_i\}.
\end{equation}
In particular, $U_{\m}=U_{\m;\m}$. We will also write $V_{\m}=V_{\m;\m}$ for simplicity.
Note that
\[
\# U_{\m;\n}=\dim\Hom_Q(M_Q(\m),\tau(M_Q(\n)))
\]
and
\[
\# V_{\m;\n}=\dim\Hom_Q(M_Q(\m),M_Q(\n)).
\]
Let $\bss_{i,j}$, $(i,j)\in V_{\m;\n}$ be the basis of $\Hom_Q(\pi_Q(M),\pi_Q(N))$ given by
\[
\bss_{i,j}(\bssaux_{k,r})=\begin{cases}
\bssauxx_{j,r}&\text{if $k=i$ and }r\in\Gamma_j,\\0&\text{otherwise.}\end{cases}
\]
Let $\altbss_{i,j}$, $(i,j)\in U_{\m;\n}$ be the basis of $\Hom_Q(\pi_Q(M),\tau(\pi_Q(N)))$ defined analogously.
Then, for any $(i,j)\in V_{\m;\n}$, $\linT_{M;N}(\bss_{i,j})$ is equal to
\begin{subequations}
\begin{equation} \label{eq: disclinT}
\sum_{k\in I\mid (k,i)\in U_{\m}\text{ and }(k,j)\in U_{\m;\n}}x_{k,i}\altbss_{k,j}
-\sum_{k\in J\mid (j,k)\in U_{\n}\text{ and }(i,k)\in U_{\m;\n}}y_{j,k}\altbss_{i,k}.
\end{equation}
Consequently, for the dual map, for any $(i,j)\in U_{\m;\n}$, $\linT_{M;N}^*(\hat\altbss_{i,j})$ is equal to
\begin{equation} \label{eq: disclinT*}
\sum_{k\in I\mid (i,k)\in U_{\m}\text{ and }(k,j)\in V_{\m;\n}}x_{i,k}\hat\bss_{k,j}
-\sum_{k\in J\mid (k,j)\in U_{\n}\text{ and }(i,k)\in V_{\m;\n}}y_{k,j}\hat\bss_{i,k}
\end{equation}
where $\hat\bss_{i,j}$, $(i,j)\in V_{\m;\n}$ and $\hat\altbss_{i,j}$, $(i,j)\in U_{\m;\n}$ denote the dual bases.
\end{subequations}

\subsection{Computational aspects}\label{sec: compute}


Given two irreducible components $C_1=\prm_Q(\m),C_2=\prm_Q(\n)\in\Irrcomp$ the computation of $\hom_\Pi(C_1,C_2)$,
(and consequently, $\ext_\Pi^1(C_1,C_2)$) and the determination of $C_1*C_2$ reduce to the following type of problem.

Let $A=A(x_1,\dots,x_k)$ be a matrix of size $n_1\times n_2$ whose entries are polynomials
in $k$ variables (say, with integer coefficients) of degree $\le d$.
We would like to determine the generic rank $r$ of $A$ efficiently (i.e., in polynomial time).
This problem is closely related to polynomial identity testing.
Unfortunately, there is no known subexponential deterministic algorithm for this problem.
However, the ``obvious'' Monte Carlo randomized algorithm works well.
More precisely, let $\varepsilon>0$
and choose integers $y_1,\dots,y_k$ independently and uniformly randomly
from a fixed set of size $s\ge d\min(n_1,n_2)\varepsilon^{-1}$.
Then, the rank of $A(y_1,\dots,y_k)$ is equal to $r$ with probability at least $1-\varepsilon$.
Indeed, if $p(x_1,\dots,x_k)$ is an $r\times r$-minor of $A$ that is not identically zero, then by the Schwartz--Zippel lemma \cite{MR594695},
the probability that $p(y_1,\dots,y_k)\ne0$ (and in particular, $\rk A(y_1,\dots,y_k)=r$) is at least
$1-\frac{\deg p}s\ge 1-\frac{dr}s$.

Recall that as in \S\ref{sec: explint} above we can represent an element of $\cn_Q(\m)$ by
coordinates $x_{i,j}$, $(i,j)\in U_{\m}$. Therefore, a ``random'' element of $\cn_Q(\m)$ (and hence of $\prm_Q(\m)$)
is represented by random coordinates.

In order to compute $\hom_\Pi(C_1,C_2)$ we use the relation \cite[(9.3)]{2103.12027} and \eqref{eq: disclinT}
to reduce it to the problem above where $k=\# U_{\m}+\# U_{\n}$, $n_1=\# U_{\m;\n}$, $n_2=\# V_{\m;\n}$
and $d=1$. In fact, the non-zero entries of $A(x_1,\dots,x_k)$ are of the form $\pm x_i$ for some $i$.

In special cases there is an efficient deterministic algorithm for the computation of $\hom_\Pi(C_1,C_2)$
in terms of $\m$, $\n$. (See Remark \ref{rem: matching} below.)
However, we do not know how to do it in general.

By the same token, using \cite[Corollary 9.3]{2103.12027}, the rigidity condition for an irreducible component $C=\prm_Q(\m)$
can be checked efficiently by a randomized algorithm (cf.\ \cite{MR3866895, MR4169050}).
Once again, it would be interesting to find a simple combinatorial criterion or at least an efficient deterministic algorithm for the rigidity of an irreducible component in general.
(See Theorem \ref{thm: LM} below for a special case.)
A closely related problem is to effectively characterize (as a subset) the open $G_V$-orbit in a given rigid irreducible component.

The binary operation $*$ gives rise to a binary operation on multisegments, also denoted by $*$,
so that $\prm_Q(\m)*\prm_Q(\n)=\prm_Q(\m*\n)$.
It would be interesting to give a purely combinatorial description of this operation on multisegments.
As part of our results we will give a recipe for this in a special case -- see Remark \ref{rem: recipebalanced}.

In general, we can determine $C_1*C_2$ (or equivalently, $\m*\n$) by a randomized algorithm.
To that end, we choose $M\in\cn_Q(\m)$, and $N\in\cn_Q(\n)$ with random coordinates.
Using \cite[Proposition 9.1]{2103.12027}, the image under $\pi_Q$ of a random extension of $M$ by $N$
is an extension $E$ of $M_Q(\m)$ by $M_Q(\n)$
which is determined by a random element of $\Ker\linT_{N;M}^*\subset\Hom_Q(M_Q(\n),\tau M_Q(\m))^*
\simeq\Ext^1_Q(M_Q(\m),M_Q(\n))$. The map $\linT_{N;M}^*$ was explicated in \eqref{eq: disclinT*}.
We can therefore recover the multisegment pertaining to $E$, as explained in \S\ref{sec: MSAn}.
Alternatively, choosing random $x_i=(T_i^+,T_i^-)\in C_i$, $i=1,2$, with $T_i^\pm:V^i\rightarrow V^i$
of degree $\pm1$, an extension of $x_1$ by $x_2$ is of the form
\[
x=(T^+,T^-)=((T_1^++T_{1,2}^+)\oplus T_2^+,(T_1^-+T_{1,2}^-)\oplus T_2^-)
\]
for a pair of linear transformations $T_{1,2}^\pm:V^1\rightarrow V^2$ of degrees $\pm1$ satisfying
\[
T_2^+T_{1,2}^-+T_{1,2}^+T_1^-=T_2^-T_{1,2}^++T_{1,2}^-T_1^+
\]
Choosing a random element in this vector space,
the $G_V$-orbit of $T^+$ gives the parameter of $C_1*C_2$ with high probability.

Similarly, given $C,C_1\in\Irrcomp$ with $C_1$ rigid, \cite[Proposition 5.1]{2103.12027} gives rise to an efficient randomized algorithm for finding the unique $C_2\in\Irrcomp$ such that $C=C_1*C_2$, if exists.
Namely, choose random points $x\in C$ and $x_1\in C_1$ and a random element $\varphi\in\Hom_\Pi(x,x_1)$ (using \cite[(9.3)]{2103.12027}).
If $\varphi$ is surjective, then with high probability $\pi_Q(\Ker\varphi)$ gives rise to the multisegment defining $C_2$.
Otherwise, $C_2$ does not exist with high probability.

\subsection{Examples} \label{sec: examPP}

We have already mentioned some simple examples of rigid modules and irreducible components in \cite[\S4]{2103.12027}.
For instance, if every segment in $\m$ is of the form $[i,i]$ (corresponding to simple roots), then
$\prm_Q(\m)$ and $\prm_{Q^{\op}}(\m)$ are rigid.
The same conclusion holds if $\Delta_i\not\prec\Delta_j$ for every segments $\Delta_i,\Delta_j$ of $\m$.

For $n\le 4$, all irreducible components are rigid (and $\Pi$ is representation-finite).

\subsubsection{Laminae and quasi-laminae} \label{sec: lad}
Let $S(r)$, $r=1,\dots,n$ be the simple $\Pi$-module whose graded dimension is $1$ in the $r$-th coordinate and $0$ elsewhere.
Let $P(r)=\Pi e_r$ be the indecomposable projective $\Pi$-module whose cosocle is $S(r)$.
By \cite[(9.9)]{2103.12027} we have
\[
\pi_Q(P(r))=M_Q(\sum_{i=1}^r[i,i+n-r]),\ \ \pi_{Q^{\op}}(P(r))=M_{Q^{\op}}(\sum_{j=1}^{n-r+1}[j,j+r-1]).
\]
The submodule lattice of $P(r)$ is isomorphic to the sublattice of the Young lattice consisting
of Young diagrams contained in the rectangle of size $r\times (n-r+1)$ (cf. \cite{MR1734897}). In particular, it is distributive.

Suppose that $1\le a_1<\dots<a_r$ and $b_1<\dots<b_r\le n$ are two sequences such that $a_i\le b_i+1$ for all $i$.
Consider a graded basis $e_j^i$, $a_i\le j\le b_i$, $i=1,\dots,r$ (graded by $j$) and homogenous maps $T_{\pm}$ of degree $\pm1$ given by
\[
T_+e_j^i=\begin{cases}e_{j+1}^i&j<b_i,\\0&\text{otherwise;}\end{cases}\ \ \ \
T_-e_j^i=\begin{cases}e_{j-1}^{i-1}&\text{if $i>1$ and $j\le b_{i-1}+1$,}\\0&\text{otherwise.}\end{cases}
\]
See diagram below for the sequences $1<3<4$ and $4<6<8$; the dots represent the basis elements and $T_\pm$ are represented by the horizontal and diagonal arrows, respectively.

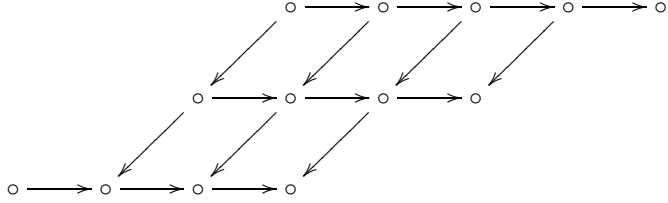
\[
\scriptscriptstyle{\xymatrix{
&&&\circ\ar[r]\ar[dl]&\circ\ar[r]\ar[dl]&\circ\ar[r]\ar[dl]&\circ\ar[r]\ar[dl]&\circ\\
&&\circ\ar[r]\ar[dl]&\circ\ar[r]\ar[dl]&\circ\ar[r]\ar[dl]&\circ&\\
\circ\ar[r]&\circ\ar[r]&\circ\ar[r]&\circ&&&& }}
\]

The resulting $\Pi$-module $M$ is a subquotient of $P(r)$. Conversely, all subquotients of $P(r)$ are obtained this way.
(The subrepresentations of $P(r)$ correspond to the case $b_i=i+n-r$ for all $i=1,\dots,r$.)
We have
\[
\pi_Q(M)=M_Q(\sum_{i=1}^r[a_i,b_i])
\]
where by convention we discard empty segments. (Such multisegments were called ladders in \cite{MR3163355}.)

Note that $M$ is indecomposable if and only if $a_{i+1}\le b_i+1$ for all $i=1,\dots,r-1$.
Following Ringel's terminology, an indecomposable subquotient of any $P(r)$ is called a \emph{lamina}.\footnote{This terminology is specific for type $A$.}
Direct sums of laminae are termed \emph{laminated modules}.
In particular, any subquotient of $P(r)$ is laminated. (We will call them \emph{quasi-lamina}e.)

Every quasi-lamina is rigid.
This follows from \cite[Lemma 8.1]{2103.12027} by an easy induction. More precisely, if $x=\prm_Q(\m)$
with $\m=[a_1,b_1]+\dots+[a_r,b_r]$ with $a_1<\dots<a_r$ and $b_1<\dots<b_r$, then we may take $x_1=\prm_Q([a_r,b_r])$
and $x_2=\prm_Q([a_1,b_1]+\dots+[a_{r-1},b_{r-1}])$ and use \cite[Corollary 9.3]{2103.12027}.
Indeed, the map $\linT_{x;x_1}$ is trivially surjective since its codomain is $0$.
The map $\linT_{x_1;x}$ is surjective since by \eqref{eq: disclinT},
on the complement of its kernel (which is at most one-dimensional)
it is represented by a square, upper triangular matrix
whose diagonal entries are coordinates of $x$ (which are free variables).

\begin{example} \label{exam: twoladsc}
Consider an irreducible component $C=\prm_Q(\m)$ of a lamina with
\[
\m=\sum_{i=1}^r[a_i,b_i]\text{ with }a_1<\dots<a_r\text{ and }b_1<\dots<b_r.
\]
For each $i$ let $t_i$ be such that $a_i-1\le t_i\le b_i$ and $t_1<\dots<t_r$. Let $C_i=\prm_Q(\m_i)$, $i=1,2$ where
\[
\m_1=\sum_{i=1}^r[a_i,t_i],\ \ \m_2=\sum_{i=1}^r[t_i+1,b_i]
\]
where we ignore any empty segments in the sums. Then, $C=C_1*C_2$.

For instance, in the case
\[
\m=[1,4]+[3,6]+[4,8],\ \m_1=[1,1]+[3,4]+[4,7],\ \m_2=[2,4]+[5,6]+[8,8],
\]
a rigid element of $C$ is represented by the following diagram (cf. \S\ref{sec: lad}).
It is an extension of a rigid element of $C_1$ by a rigid element of $C_2$ where the former (resp., latter) is represented by the
blue (resp., red) dots and arrows.
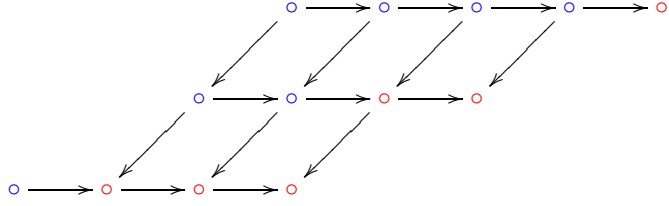
\[
\scriptscriptstyle{\xymatrix{
&&&{\color{blue}{\circ}}\ar@[blue][r]\ar@[blue][dl]&{\color{blue}{\circ}}\ar@[blue][r]\ar@[blue][dl]&
{\color{blue}{\circ}}\ar@[blue][r]\ar[dl]&{\color{blue}{\circ}}\ar[r]\ar[dl]&{\color{red}{\circ}}\\
&&{\color{blue}{\circ}}\ar@[blue][r]\ar[dl]&{\color{blue}{\circ}}\ar[r]\ar[dl]&{\color{red}{\circ}}\ar@[red][r]\ar@[red][dl]&{\color{red}{\circ}}&\\
{\color{blue}{\circ}}\ar[r]&{\color{red}{\circ}}\ar@[red][r]&{\color{red}{\circ}}\ar@[red][r]&{\color{red}{\circ}}&&&& }}
\]
\end{example}

\begin{remark} \label{rem: matching}
Let $C_1=\prm_Q(\m)$ and $C_2=\prm_Q(\n)$ with $\m=\sum_{i\in I}\Delta_i$ and $\n=\sum_{j\in J}\Gamma_j$.
As in \cite[\S 6.3]{MR3573961}, let
\[
\grph_{\m;\n}=(U_{\m;\n},V_{\m;\n},E_{\m;\n})
\]
be the bipartite graph where $U_{\m;\n}$ and $V_{\m;\n}$ are as in \eqref{eq: UandV} and
\[
\begin{split}
E_{\m;\n}=&\{((i,k),(i,j))\in U_{\m;\n}\times V_{\m;\n}\mid (j,k)\in U_{\n}\}\cup\\
&\{((i,k),(j,k))\in U_{\m;\n}\times V_{\m;\n}\mid (i,j)\in U_{\m}\}.
\end{split}
\]
Denote by $\mtch(\m,\n)$ the condition that $\grph_{\m;\n}$ admits a matching that covers all vertices of $U_{\m;\n}$.\footnote{Recall
that a matching in a graph is a set $\mtching$ of edges, no two of which have a vertex in common.
The vertices covered by $\mtching$ are those which belong to one of the edges in $\mtching$.}
As noted in \cite{MR4169050}, $\mtch(\n,\m)$ is a necessary condition for
$\prm_Q(\m)*\prm_Q(\n)=\prm_Q(\m+\n)$ (cf.\ \cite[(9.4)]{2103.12027} and \eqref{eq: disclinT}).
This condition is not sufficient in general. However, suppose that $C_1$ or $C_2$ is a quasi-lamina.
Then, the conditions $\prm_Q(\m)*\prm_Q(\n)=\prm_Q(\m+\n)$ and $\mtch(\n,\m)$ are equivalent by \cite[Proposition 6.20]{MR3573961}.
In particular, $C_1$ and $C_2$ commute if and only if both $\mtch(\m,\n)$ and $\mtch(\n,\m)$
are satisfied. This criterion is very useful in practice. We will use it in \S\ref{sec: endofproof} where we prove
our main result. (See Lemma \ref{lem: combreg}.)

More generally, let $\mtching_{\m;\n}$ be a matching of $\grph_{\m;\n}$ of maximal size. Then,
using the argument of \cite[\S 6.4]{MR3573961} and \cite[\S 8.3]{MR4169050} one can show that
if $C_1$ or $C_2$ is a quasi-lamina, then $\min_{(M,N)\in \cn_Q(\m)\times\cn_Q(\n)}\dim\Coker\linT_{N;M}$
is equal to the number of vertices in $U_{\m;\n}$ that are not covered by $\mtching_{\m;\n}$.
Equivalently, $\hom_\Pi(C_2,C_1)$ is equal to the number of
vertices of $V_{\m;\n}$ that are not covered by $\mtching_{\m;\n}$.
In particular, since the size of $\mtching_{\m;\n}$ can be computed efficiently by a deterministic algorithm,
the same is true for $\hom_\Pi(C_1,C_2)$ and $\ext^1_\Pi(C_1,C_2)$ in this case.
We omit the details.
\end{remark}

\subsubsection{Non-rigid components}

\begin{example} \label{ex: nonrigid}
An example of a non-rigid irreducible component for $n=5$ was analyzed in detail in \cite{MR2115084}, following \cite{MR1959765}.
Consider $\bfd=(1,2,2,2,1)$, so that $\dim\Lambda(V)=12$ and $\dim G_V=14$.
Let
\[
C=\prm_Q(\m)=\prm_{Q^{\op}}(\m)\text{ where }\m=[4,5]+[2,4]+[3,3]+[1,2].
\]
Then, $C$ is the closure of a one-parameter family $\fl$ of $11$-dimensional orbits.
For any $x\in\fl$ we have $\dim\End_{\Pi}(x)=3$ and $\dim\Ext^1_{\Pi}(x,x)=2$, while
for any two distinct $x,y\in\fl$ we have $\dim\Hom_{\Pi}(x,y)=2$ and $\Ext^1_{\Pi}(x,y)=0$.

Thus, by \cite[Corollary 3.3]{2103.12027}, $C\oplus C$ is an irreducible component even though for any $x\in\fl$ there is a short exact sequence
\[
0\rightarrow x\rightarrow P(2)\oplus P(4)\rightarrow x\rightarrow 0.
\]
\end{example}

\begin{example} \label{ex: CCrigid}
As was pointed out in \cite[Remark 4.3]{2103.12027}, there exist non-rigid components $C$ for which $C*C$ is rigid.
An example for $n=7$ is
\[
C=\prm_Q([4,7]+[5,6]+[2,5]+[3,4]+[1,3])
\]
for which
\[
C*C=\prm_Q([4,7]+[2,7]+[5,6]+[3,6]+[4,5]+[1,5]+[2,4]+[3,3]+[1,3]).
\]
\end{example}

\subsubsection{}
Finally, we explicate a special case of \cite[Lemma 9.5]{2103.12027} (cf.\ \cite[Remark 8.5]{2103.12027}).

\begin{lemma}
Suppose that $\m=\Delta_1+\dots+\Delta_k$ where the $\Delta_i$'s are distinct.
Assume that $C=\prm_Q(\m)$ is rigid. Then, we can write non-trivially $C=C_1*C_2$
where $C_1$ and $C_2$ are rigid and $\ext_\Pi^1(C,C_1)=0$.
\end{lemma}

\begin{proof}
Indeed, without loss of generality we may assume that $\Delta_i\not\prec\Delta_k$ and $\Delta_i\not\prec\rshft\Delta_k$
for all $i<k$. (For instance, this holds if $b(\Delta_k)+e(\Delta_k)\le b(\Delta_i)+e(\Delta_i)$ for all $i<k$.)
Let $x\in C$ be rigid. Then, $\pi_Q(x)\simeq M_Q(\m)$. Write $\m=\m'+\m''$ where $\m'=\Delta_1+\dots+\Delta_{k-1}$
and $\m''=\Delta_k$. Then, the conditions of all parts of \cite[Lemma 9.5]{2103.12027} are satisfied
for the decomposition $M_Q(\m)=M_Q(\m')\oplus M_Q(\m'')$. Taking $C_1=\prm_Q(\m')$ and $C_2=\prm_Q(\m'')$, the result follows.
\end{proof}

\section{Relation to representation theory of \texorpdfstring{$\GL$}{GL}} \label{sec: GLn}

From now on, let $F$ be a local non-archimedean field with normalized absolute value $\abs{\cdot}$.
Let
\[
\Reps=\oplus_{m\ge0}\Reps(\GL_m(F))
\]
be the category of complex, smooth representations of finite length of $\GL_m(F)$, $m\ge0$.
This is a ring category with tensor functor given by (normalized) parabolic induction $\pi_1\times\pi_2$
with respect to the parabolic subgroup of block upper triangular matrices.
The unit object $\one$ is the one-dimensional representation of $\GL_0(F)=1$.

By Zelevinsky's classification \cite{MR584084}, there is a bijection
\[
\m\rightarrow Z(\m)
\]
between $\Mult^n$ and the set $\Irr=\Irr^n$ of irreducible subquotients (up to isomorphism) of
\[
\abs{\cdot}^{a_1}\times\dots\times\abs{\cdot}^{a_k}
\]
(a representation of $\GL_k(F)$) where $a_1,\dots,a_k$ range over all finite sequences of integers between $1$ and $n$.
Under this bijection, for any $\bfd\in\N^n$, $\Mult(\bfd)$ corresponds to the set $\Irr(\bfd)$ of irreducible subquotients of
\[
\overbrace{\abs{\cdot}\times\dots\times\abs{\cdot}}^{\bfd(1)}\times\dots\times
\overbrace{\abs{\cdot}^n\times\dots\times\abs{\cdot}^n}^{\bfd(n)}.
\]

More precisely, given a segment $\Delta=[a,b]$ define
\[
Z(\Delta)=\soc(\abs{\cdot}^a\times\dots\times\abs{\cdot}^b),\ L(\Delta)=\soc(\abs{\cdot}^b\times\dots\times\abs{\cdot}^a).
\]
Thus, $Z(\Delta)$ is the one-dimensional character $\abs{\det\cdot}^{(a+b)/2}$ of $\GL_{b-a+1}(F)$
and $L(\Delta)$ is the twist of the Steinberg representation of $\GL_{b-a+1}(F)$ by $\abs{\det\cdot}^{(a+b)/2}$.
Given a multisegment $\m$ we write it as $\m=\Delta_1+\dots+\Delta_r$ in such a way that $\Delta_i\not\prec\Delta_j$ for all $i<j$.
Then, up to isomorphism, the standard module representation
\begin{equation} \label{eq: std}
\std(\m)=Z(\Delta_1)\times\dots\times Z(\Delta_r)
\end{equation}
depends only on $\m$ and not on the choice of permissible ordering of the $\Delta_i$'s and
\[
Z(\m):=\soc(\std(\m)).
\]
We also have a bijection
\[
\m\mapsto L(\m):=\soc(L(\Delta_r)\times\dots\times L(\Delta_1))
\]
where $\m=\Delta_1+\dots+\Delta_r$, ordered as above.

Let $\BZ$ be the Bernstein--Zelevinsky Hopf algebra over $\Z$, namely, the Grothendieck group of $\Reps$ with
product and coproduct induced from the tensor functor (i.e., parabolic induction) and Jacquet functor, respectively.
Then, $\BZ$ is a (commutative) polynomial ring over $\Z$ freely generated by the so-called segment representations.
The latter are obtained from Zelevinsky's more general notion of segments \cite[\S3]{MR584084}.
If we limit ourselves to our restricted notion of segments (indexed by pairs of integers $i\le j$ in $\{1,\dots,n\}$), then
we get a Hopf subalgebra $\BZ^n$ (a polynomial ring over $\Z$ freely generated by $Z([i,j])$, $1\le i\le j\le n$).
The irreducible representations $Z(\m)$, $\m\in\Mult^n$ form a $\Z$-basis for $\BZ^n$.
We caution that $\BZ^n$ should not be confused with the graded part of $\BZ$ corresponding to the Grothendieck group of $\GL_n(F)$. If $\pi$ is an irreducible representation of $\GL_k(F)$ we will write $\deg\pi=k$.

Define a bijection
\[
\Irrcomp\rightarrow\Irr:\ \ \ C\mapsto\pi(C)
\]
by
\begin{equation} \label{eq: zmpi}
Z(\m)=\pi(\prm_Q(\m)),
\end{equation}
for every $\m\in\Mult$. By \eqref{eq: prm*} and \cite{MR863522} we also have
\[
L(\m)=\pi(\prm_{Q^{\op}}(\m))\ \ \forall\m\in\Mult.
\]
Clearly, $\pi(\Irrcomp(\bfd))=\Irr(\bfd)$ for every $\bfd\in\N^n$.

\begin{remark}
Although in this paper we will only use the bijection $C\mapsto\pi(C)$ as a bookkeeping device, it lies much deeper.
One can identify $\BZ^n\otimes\C$ with the Hopf algebra $\C[N]$ of regular functions on
the maximal unipotent subgroup $N$ of $\GL_{n+1}$ of upper unitriangular matrices.
Under this identification, the class of $Z([i,j])$ becomes the $(i,j)$-coordinate function.
On the other hand, $\C[N]$ is also the dual of the universal enveloping algebra $U(\operatorname{Lie} N)$ (a cocommutative Hopf algebra),
i.e., the positive part of $U(\mathfrak{sl}_{n+1})$.
In turn, $U(\operatorname{Lie} N)$ is endowed with Lusztig's canonical basis \cite{MR1035415}.
Under the identification above, the dual canonical basis in $\C[N]$ coincides with the classes of
irreducible representations in $\BZ^n$.
The bijection $C\mapsto\pi(C)$ realizes this fact, taking into account that the canonical basis is parameterized by the set
\[
\Irrcomp=\Irrcomp^n=\coprod_{\bfd\in\N^n}\Irrcomp(\bfd)
\]
(see \cite{MR1088333}) -- this parameterization is compatible with the $\N^n$-grading of $U(\operatorname{Lie} N)$.
See \cite{MR1985725, MaxChari2021} and the references therein for more details.
\end{remark}

For any irreducible component $C$ we have
\begin{equation} \label{eq: contrag}
\pi(C)^\vee\otimes\abs{\cdot}^{n+1}=\pi(C^\vee)
\end{equation}
where on the left-hand side, $^\vee$ denotes the contragredient and we twist by the character $\abs{\det}^{n+1}$.

We also remark that for any $\pi_1,\pi_2,\sigma\in\Irr$ we have (cf. \cite[p. 173]{MR863522})
\begin{equation} \label{eq: switch}
\sigma\hookrightarrow\pi_1\times\pi_2\iff\sigma^\vee\hookrightarrow\pi_2^\vee\times\pi_1^\vee\iff
\pi_2\times\pi_1\twoheadrightarrow\sigma.
\end{equation}

\begin{remark} \label{rem: simsub}
For any $\m_1,\m_2\in\Mult$, $Z(\m_1+\m_2)$ occurs in the Jordan--H\"older sequence of $Z(\m_1)\times Z(\m_2)$
(in fact, with multiplicity one).
In particular, if $Z(\m_1)\times Z(\m_2)$ is irreducible, then it is equal to $Z(\m_1+\m_2)$.
Thus, by \cite*[(9.5)]{2103.12027}, if $C_1,C_2\in\Irrcomp$ strongly commute and $\pi(C_1)\times\pi(C_2)$ is irreducible,
then it is equal to $\pi(C_1*C_2)=\pi(C_1\oplus C_2)$.
\end{remark}

\begin{example}
Going back to Example \ref{ex: nonrigid}, let $C=\prm_Q(\m)$, $C_1=\prm_Q(\m_1)$, $C_2=\prm_Q(\m_2)$ where
\[
\m=[4,5]+[2,4]+[3,3]+[1,2],
\]
\[
\m_1=[1,4]+[2,5],\ \ \m_2=[1,2]+[2,3]+[3,4]+[4,5].
\]
Then, by \cite{MR1959765},
\[
\pi(C)\times\pi(C)=\pi(C\oplus C)+\pi(C_1)\times\pi(C_2)=
\pi(C\oplus C)+\pi(C_1\oplus C_2).
\]
Note that $C\oplus C=\prm_Q(\m+\m)$ and $C_1\oplus C_2=\prm_Q(\m_1+\m_2)$.
\end{example}

\subsection*{Conjectures of Geiss and Schr\"oer} \label{sec: conjGS}

The following is for the most part, a special case of a conjecture of Geiss and Schr\"oer for type $A$.

\begin{subtheorem}{theorem}
\begin{conjecture}[cf. Conjecture 5.3 of \cite{MR2115084}] \label{conj: GS}
Let $C_i\in\Irrcomp(\bfd_i)$, $i=1,2$.
Then, $\pi(C_1)\times\pi(C_2)$ is irreducible if and only if
\begin{equation} \label{part: condirr}
\begin{split}
\text{there exist nonempty open subsets $U_i\subset C_i$, $i=1,2$}\\\text{ such that }
\Ext_\Pi^1(x_1,x_2)=0\text{ for all }x_1\in U_1, x_2\in U_2.
\end{split}
\end{equation}
In particular (by \cite[Corollary 3.3]{2103.12027}) if $\pi(C_1)\times\pi(C_2)$ is irreducible,
then $C_1$ and $C_2$ strongly commute (and hence $\pi(C_1)\times\pi(C_2)=\pi(C_1\oplus C_2)$ by Remark \ref{rem: simsub}).
\end{conjecture}

In fact, the original formulation in [ibid.] (which is attributed in part to Marsh and Reineke)
is pertaining to the product of two elements in the dual canonical basis of the negative part of the quantized enveloping algebra of
the Lie algebra corresponding to a simply laced Dynkin diagram.
For type $A$, \cite[Conjecture 5.3]{MR2115084} amounts to the ``if'' part of Conjecture \ref{conj: GS} as well as a related,
though not equivalent, variant of the ``only if'' part.

In general, the strong commutativity of $C_1$ and $C_2$ does not imply \eqref{part: condirr}
even if $C_1=C_2$ (see Example \ref{ex: nonrigid}).
However, if $C_1$ or $C_2$ is rigid, then Conjecture \ref{conj: GS} specializes as follows.

\begin{conjecture} \label{conj: GS2}
Let $C_i\in\Irrcomp(\bfd_i)$, $i=1,2$. Suppose that $C_1$ or $C_2$ is rigid.
Then, $\pi(C_1)\times\pi(C_2)$ is irreducible if and only if $C_1$ and $C_2$ commute
(see \cite[Proposition 6.1]{2103.12027}).
\end{conjecture}

Recall that in general, strong (as well as weak) commutativity can be checked efficiently by a randomized algorithm.

Some form of Conjecture \ref{conj: GS2} appears as Conjecture 1 in \cite{MR4169050} -- we were unaware of
\cite[Conjecture 5.3]{MR2115084} at the time.

Finally, we also single out the case $C_1=C_2$, in which Conjecture \ref{conj: GS} specializes as follows.
We say that $\pi$ is $\square$-irreducible if $\pi\times\pi$ is irreducible.
\begin{conjecture} \label{conj: GSR}
For any $C\in\Irrcomp(\bfd)$
\[
\pi(C)\text{ is $\square$-irreducible $\iff C$ is rigid.}
\]
\end{conjecture}
\end{subtheorem}

This conjecture was stated in \cite{MR3866895}, albeit with an inaccurate attribution to a later reference.

At any rate, as far as we know, Conjectures \ref{conj: GS2} and \ref{conj: GSR} (let alone Conjecture \ref{conj: GS}) are wide open in general, even in one direction.
We also remark that if neither $C_i$ is rigid and $C_1\ne C_2$, then it is unclear how to check condition \eqref{part: condirr} of Conjecture \ref{conj: GS} algorithmically.

\section{\texorpdfstring{$\square$}{Square}-irreducible representations} \label{sec: sqr}

For any two representations of finite length $\pi$ and $\sigma$ of $\GL$ we denote by
\[
\inter_{\pi;\sigma}:\pi\times\sigma\rightarrow\sigma\times\pi
\]
the normalized intertwining operator obtained by taking the leading term in the Laurent expansion
of the unnormalized intertwining operator at $s=0$. (See \cite{MR3866895} for more details.)

We recall the following results, which are adaptations of basic results of Kang, Kashiwara, Kim and Oh.

\begin{subtheorem}{theorem}
\begin{theorem}[\cite{MR3866895}, after \cite{MR3314831}] \label{thm: kkko}
The following conditions are equivalent for a representation $\pi$ of finite length of $\GL_r(F)$.
\begin{enumerate}
\item $\pi$ is $\square$-irreducible.
\item $\End_{\GL_{2r}(F)}(\pi\times\pi)=\C$.
\item $\inter_{\pi;\pi}$ is a scalar.
\end{enumerate}
\end{theorem}

Let us say that a representation $\tau$ of finite length is \SI{} if its socle is irreducible and occurs with multiplicity
one in the Jordan--H\"older sequence of $\tau$.

\begin{theorem}[\cite{MR3866895}, after \cite{MR3314831}] \label{thm: kkko2}
Let $\pi_1$ and $\pi_2$ be two irreducible representations of $\GL$.
Suppose that at least one of them is $\square$-irreducible.
Then,
\begin{enumerate}
\item $\pi_1\times\pi_2$ and $\pi_2\times\pi_1$ are \SI.
\item $\soc(\pi_1\times\pi_2)$ is the image of the normalized intertwining operator $\pi_2\times\pi_1\rightarrow\pi_1\times\pi_2$.
\item \label{part: irrcit} $\pi_1\times\pi_2$ is irreducible if and only if $\soc(\pi_1\times\pi_2)\simeq\soc(\pi_2\times\pi_1)$.
\item If $\pi_1\times\pi_2$ is irreducible and both $\pi_1$ and $\pi_2$ are $\square$-irreducible, then $\pi_1\times\pi_2$ is $\square$-irreducible.
\end{enumerate}
\end{theorem}

Note that part \ref{part: irrcit} was not stated explicitly in \cite{MR3866895} but it follows easily for the other parts.
Indeed, if $\sigma:=\soc(\pi_1\times\pi_2)\simeq\soc(\pi_2\times\pi_1)$, then by \eqref{eq: switch} $\sigma$ occurs both as a subrepresentation
and as a quotient of $\tau:=\pi_1\times\pi_2$. Since the multiplicity of $\sigma$ in the Jordan--H\"older sequence of $\tau$
is one, this means that $\sigma$ is a direct summand of $\tau$. However, since $\sigma=\soc(\tau)$, this implies
that $\tau=\sigma$.

\begin{corollary} \label{cor: allprod}
Let $\pi_1,\dots,\pi_k$ be irreducible representations such that at most one of them is not $\square$-irreducible.
Assume that $\pi_i\times\pi_j$ is irreducible for all $i\ne j$. Then, $\pi_1\times\dots\times\pi_k$ is irreducible.
\end{corollary}

This follows by induction on $k$, by analyzing the normalized intertwining operator.
See \cite{MR4272560} for a more general statement.

Denote by $\cos$ the cosocle of a representation.
\begin{theorem}[cf.\ Corollary 3.7 of \cite{MR3314831}] \label{thm: cancel}
Let $\pi$ be a $\square$-irreducible representation of $\GL$.
Then, the two functions
\begin{align*}
\sigma\mapsto\soc(\sigma\times\pi)=\cos(\pi\times\sigma)\\
\sigma\mapsto\soc(\pi\times\sigma)=\cos(\sigma\times\pi)
\end{align*}
are injective on the set of irreducible representations of $\GL$.
\end{theorem}

For convenience, we adapt the proof of Proposition 3.6 and Corollary 3.7 of \cite{MR3314831}
to the setting of representation theory of $\GL$.

\begin{proof}
We first note that it is enough to prove the injectivity of the first map, since the injectivity
of the second one would then follow by passing to the contragredient.

Let $r=\deg\pi$ and let $J_\pi$ be the left-adjoint of the functor $\sigma\mapsto\sigma\times\pi$.
Thus, for any representation $\tau$ of $\GL_s(F)$,
$J_\pi(\tau)$ is the $(\GL_r(F),\pi)$-coinvariants (in the second factor)
of the (normalized) Jacquet module $J(\tau)$ of $\tau$
with respect to the parabolic subgroup of type $(s-r,r)$ (interpreted as $0$ if $s<r$).

Suppose now that $\tau$ is irreducible and $\omega:=J_\pi(\tau)\ne0$. Let
\[
\iota:\tau\rightarrow\omega\times\pi
\]
be the canonical morphism corresponding to the identity morphism of $\omega$.
Since the latter does not factor through any proper submodule of $\omega$,
\begin{equation} \label{eq: notprop}
\iota(\tau)\text{ is not contained in $\kappa\times\pi$ for any proper subrepresentation }\kappa\subsetneq\omega.
\end{equation}
Moreover, since $\tau$ is irreducible, $\iota$ is a monomorphism.

Next, we claim that
\begin{equation} \label{eq: impiom}
\text{the image of the intertwining operator }
\inter_{\pi;\omega}:\pi\times\omega\rightarrow\omega\times\pi
\text{ is }\iota(\tau).
\end{equation}

Recall that the restriction of
\[
\inter_{\pi;\omega\times\pi}:\pi\times\omega\times\pi\rightarrow\omega\times\pi\times\pi
\]
to $\pi\times\tau$ (via $\iota$) is either $0$ or $\inter_{\pi;\tau}$, and in any case
its image is contained in $\tau\times\pi$ (see \cite[Lemma 2.3(3)]{MR3866895}). On the other hand,
\[
\inter_{\pi;\omega\times\pi}=(\id_\omega\times\inter_{\pi;\pi})\circ(\inter_{\pi;\omega}\times\id_{\pi})
\]
which by assumption, is a non-zero scalar multiple of $\inter_{\pi;\omega}\times\id_{\pi}$. Thus, we get a commutative diagram
\[
  \begin{tikzcd}
    \pi\times\tau \arrow{r} \arrow[hookrightarrow]{d}{\id_\pi\times\iota} &
    \tau\times\pi \arrow[hookrightarrow]{d}{\iota\times\id_\pi} \\
    \pi\times\omega\times\pi \arrow{r}{\inter_{\pi;\omega}\times\id_{\pi}} & \omega\times\pi\times\pi
  \end{tikzcd}
\]
This means that
\[
\pi\times\iota(\tau)\subset\inter_{\pi;\omega}^{-1}(\iota(\tau))\times\pi.
\]
Hence, by \cite[Corollary 2.2]{MR3866895}, there exists a subrepresentation $\kappa$ of $\omega$ such that
$\inter_{\pi;\omega}(\pi\times\kappa)\subset\iota(\tau)$ and $\iota(\tau)\subset\kappa\times\pi$.
By \eqref{eq: notprop}, $\kappa=\omega$, which implies \eqref{eq: impiom}.

It follows that we can factorize $R_{\pi;\omega}$ as $\iota\circ\eta$ where
\[
\eta:\pi\times\omega\rightarrow\tau.
\]
We claim that for any proper subrepresentation $\kappa$ of $\omega$ we have $\eta\rest_{\pi\times\kappa}\equiv0$.
Indeed, otherwise, the restriction of $R_{\pi;\omega}$ to $\pi\times\kappa$ would be non-zero, and hence equal to $R_{\pi;\kappa}$.
This would contradict \eqref{eq: notprop} and \eqref{eq: impiom}.

We infer that $\omega$ has a unique maximal (proper) subrepresentation, i.e., $\cos(\omega)$ is irreducible.

Finally, if $\sigma$ is an irreducible representation of $\GL$ and $\tau=\soc(\sigma\times\pi)$, then the relation
\[
\Hom(\omega,\sigma)=\Hom(\tau,\sigma\times\pi)
\]
implies that $\sigma=\cos(\omega)$. In particular, $\sigma$ is determined by $\tau$.
\end{proof}

\end{subtheorem}

\section{A new conjecture} \label{sec: mainconj}

We make the following conjecture which is inspired by \cite[Conjecture 5.3]{MR2115084}.

\begin{conjecture} \label{conj: subconj}
$\pi(C_1*C_2)$ is a subrepresentation of $\pi(C_1)\times\pi(C_2)$.
In particular, if $\pi(C_1)$ or $\pi(C_2)$ is $\square$-irreducible, then $\pi(C_1*C_2)=\soc(\pi(C_1)\times\pi(C_2))$.
\end{conjecture}

Roughly speaking, Conjecture \ref{conj: subconj} says that a generic extension of a generic $x_1\in C_1$ by a generic $x_2\in C_2$
determines an irreducible subrepresentation of $\pi(C_1)\times\pi(C_2)$.

A good example to bear in mind is the following.

\begin{example} \label{exam: twolads}
As in Example \ref{exam: twoladsc} let $C=\prm_Q(\m)$ be a lamina with
\[
\m=\sum_{i=1}^r[a_i,b_i]\text{ with }a_1<\dots<a_r\text{ and }b_1<\dots<b_r.
\]
For each $i$ let $t_i$ be such that $a_i-1\le t_i\le b_i$ and $t_1<\dots<t_r$.
Then, $C=C_1*C_2$ where
\[
C_1=\prm_Q(\sum_{i=1}^r[a_i,t_i])\text{ and }C_2=\prm_Q(\sum_{i=1}^r[t_i+1,b_i]).
\]
(As usual, we ignore any empty segments in the sums.)
By \cite{MR2996769}, we have $\pi(C)=\soc(\pi(C_1)\times\pi(C_2))$.
\end{example}

A weak form of Conjecture \ref{conj: subconj} is that $\pi(C_1*C_2)$ occurs as a subquotient
of $\pi(C_1)\times\pi(C_2)$.
In fact, this would follow from the first part of \cite[Conjecture 5.3]{MR2115084} (for type $A$)
together with \cite[Theorem 3.1]{2103.12027}.
However, even this weaker conjecture is open as far as we know
(even assuming that $C_1$ or $C_2$ is rigid).

The decomposition of $\pi(C)\times\pi(C)$ in the Grothendieck group was computed in \cite{1911.04270}
using computer calculations for $C$ of the form \eqref{eq: Cw} below, for $n\le 6$. In these cases,
\begin{enumerate}
\item The multiplicity of $\pi(C*C)$ in the Jordan--H\"older sequence of $\pi(C)\times\pi(C)$ is always one or two.
\item This multiplicity is one if and only if $C$ strongly commutes with itself if and only if $\pi(C)\times\pi(C)$ is of length at most two.
\end{enumerate}

Some consequences of Conjecture \ref{conj: subconj} and its interrelation with Conjectures \ref{conj: GS}--\ref{conj: GSR} are
stated in the following

\begin{lemma} \label{lem: conjcons}
Assume that Conjecture \ref{conj: subconj} holds.
Then, for any $\m,\n\in\Mult$
\begin{enumerate}
\item If $\prm_Q(\m+\n)=\prm_Q(\m)*\prm_Q(\n)$ (cf. \cite[(9.4)]{2103.12027}), then
$Z(\m+\n)\hookrightarrow Z(\m)\times Z(\n)$.
\item If $\prm_Q(\m)$ and $\prm_Q(\n)$ strongly commute, then
$Z(\m+\n)$ is a direct summand of $Z(\m)\times Z(\n)$. In particular, if $Z(\m)$ or $Z(\n)$ is $\square$-irreducible,
then $Z(\m)\times Z(\n)$ is irreducible.
\item If $\prm_Q(\m)$ and $\prm_Q(\n)$ do not strongly commute, then $Z(\m)\times Z(\n)$ is reducible.
\end{enumerate}
\end{lemma}

Indeed, the first part follows from \eqref{eq: zmpi}.
The last two parts follow from \cite*[(9.5)]{2103.12027}, \eqref{eq: switch} and Remark \ref{rem: simsub}.
(For the last part, the weaker form of Conjecture \ref{conj: subconj} suffices.)

Note that in the case where $Z(\m)$ or $Z(\n)$ is $\square$-irreducible, the first part is precisely Conjecture 2 of \cite{MR4169050}. Thus, Conjecture \ref{conj: subconj} is a strengthening of \cite[Conjecture 2]{MR4169050}.

If we assume Conjecture \ref{conj: GSR}, then Conjecture \ref{conj: subconj} also gives a randomized algorithm
to determine the multisegment $\m_2$ (if exists) such that $Z(\m)=\soc(Z(\m_1)\times Z(\m_2))$
given two multisegments $\m$, $\m_1$ such $Z(\m_1)$ is $\square$-irreducible.
(See Theorem \ref{thm: cancel} and \S\ref{sec: compute}.)

An easy case of Conjecture \ref{conj: subconj} is the following.

\begin{lemma} \label{lem: notprec1}
Let $\m_1,\m_2$ be two multisegments and let $\m=\m_1+\m_2$.
Assume that $\Ext_Q^1(M_Q(\m_1),M_Q(\m_2))=0$ (cf. \eqref{eq: notprec}).
Let $C_i=\prm_Q(\m_i)$, $i=1,2$ and $C=\prm_Q(\m)$. Then,
$C_1*C_2=C$ and $\soc(\pi(C_1)\times\pi(C_2))=\pi(C)$.
The same holds if we replace $Q$ by $Q^{\op}$ throughout.
\end{lemma}

\begin{proof}
If $\Ext_Q^1(M_Q(\m_1),M_Q(\m_2))=0$, then any extension of $x_1\in\cn_Q(\m_1)$
by $x_2\in\cn_Q(\m_2)$ lies in $\cn_Q(\m)$.
Hence, $C=C_1*C_2$ by \cite[(3.3)]{2103.12027}. The relation
\[
Z(\m)=\soc(Z(\m_1)\times Z(\m_2))
\]
(under the above condition) is well known and follows from the fact that
\[
\std(\m)=\std(\m_1)\times\std(\m_2).
\]
The second part is similar.
\end{proof}

The following analogue of \cite[Lemma 8.1]{2103.12027} gives another piece of evidence for Conjecture \ref{conj: subconj}.

\begin{lemma} \label{lem: sqr123}
Let $\pi_1$, $\pi_2$ be irreducible representations of $\GL$ with $\pi_2$ $\square$-irreducible.
Let $\pi=\soc(\pi_1\times\pi_2)$. Assume that $\pi\times\pi_1$ is irreducible.
Then, $\pi_1$ and $\pi$ are $\square$-irreducible.
\end{lemma}

\begin{proof}
The $\square$-irreducibility of $\pi$ was proved in \cite[Lemma 2.10]{MR3866895}.
To show that $\pi_1$ is also $\square$-irreducible, consider the morphisms
\[
\pi_2\times\pi_1\times\pi_1\xrightarrow{f}\pi_1\times\pi_2\times\pi_1\xrightarrow{g}\pi_1\times\pi_1\times\pi_2
\]
where $f=R_{\pi_2;\pi_1}\times\id_{\pi_1}$ and $g=\id_{\pi_1}\times R_{\pi_2;\pi_1}$.
The image of $f$ is $\pi\times\pi_1$. The composition $h=g\circ f$ is non-zero
since $\pi\times\pi_1\not\subset\pi_1\times\omega$ where $\omega=\Ker R_{\pi_2;\pi_1}$ by \cite[Corollary 2.2]{MR3866895}.
Therefore, by \cite[Lemma 2.3]{MR3866895}, $h$ is equal to
a nonzero scalar multiple of $R_{\pi_2;\pi_1\times\pi_1}$.

The image of $g$ is $\pi_1\times\pi$, which is irreducible by assumption.
Therefore, the image of $h$ coincides with $\pi_1\times\pi$.

Let $\tau$ be an irreducible subrepresentation of $\pi_1\times\pi_1$.
The restriction $f'$ of $f$ to $\pi_2\times\tau$ is nonzero since
$\pi_2\times\tau\not\subset\omega\times\pi_1$, again by \cite[Corollary 2.2]{MR3866895}.
Since the image of $f$ is irreducible, it coincides with the image of $f'$.
Let $h'=g\circ f'$ be the restriction of $h$ to $\pi_2\times\tau$.
Then, the image of $h'$ is equal to the image of $h$. In particular, $h'$ is non-zero.
Therefore, by \cite[Lemma 2.3]{MR3866895}, $h'$ is a non-zero scalar multiple of $R_{\pi_2;\tau}$.
It follows that the image of $h'$ is contained in $\tau\times\pi_2$.
We obtain $\pi_1\times\pi\subset\tau\times\pi_2$, and therefore $\tau=\pi_1\times\pi_1$, as before.
We conclude that $\pi_1$ is $\square$-irreducible.
\end{proof}

\begin{definition}
Let $C\in\Irrcomp$. We say that $C$ is \emph{good} if for every $D\in\Irrcomp$ we have
\[
\pi(C*D)\hookrightarrow\pi(C)\times\pi(D)\text{ and }\pi(D*C)\hookrightarrow\pi(D)\times\pi(C).
\]
In this case, if $C$ is rigid, we also say that any rigid element of $C$ is good.
\end{definition}
Conjecture \ref{conj: subconj} states that every $C\in\Irrcomp$ is good.

We hasten to say that at the moment, we do not know how to prove that a non-rigid irreducible component is good even in a single example.

\begin{remark} \label{rem: symm}
Suppose that $\Cls$ is a set of irreducible components that is invariant under $C\mapsto C^\vee$.
Then, in order to prove that every $C\in\Cls$ is good, it suffices to show that
$\pi(C*D)\hookrightarrow\pi(C)\times\pi(D)$ for $C\in\Cls$ and $D\in\Irrcomp$.
Indeed, by \eqref{eq: switch}, \eqref{eq: contrag} and \eqref{eq: cont*}
\begin{align*}
\pi(D*C)\hookrightarrow\pi(D)\times\pi(C)\iff&
\pi(D*C)^\vee\hookrightarrow\pi(C)^\vee\times\pi(D)^\vee\\\iff&
\pi(C^\vee*D^\vee)\hookrightarrow\pi(C^\vee)\times\pi(D^\vee).
\end{align*}
\end{remark}

Another sanity check for Conjecture \ref{conj: subconj} is the following

\begin{lemma} \label{lem: prod}
Suppose that $C_1,\dots,C_k\in\Irrcomp$ are good and pairwise strongly commute
and that $\pi(C_i)$ is $\square$-irreducible for all $i$.
Then, $C_1*\dots*C_k$ is good.\footnote{When there is no danger of confusion we will write
$C_1*\dots *C_k$ for the ``product'' in arbitrary order.}
\end{lemma}

\begin{proof}
By induction, it is enough to consider the case $k=2$.
Let $C_i\in\Irrcomp$, $i=1,2,3$ and suppose that $C_1$ and $C_2$ are good and strongly commute.
Let $\pi_i=\pi(C_i)$, $i=1,2,3$ and suppose that $\pi_1$ and $\pi_2$ are $\square$-irreducible.
By assumption and Theorem \ref{thm: kkko2}, $\pi_1\times\pi_2$ is $\square$-irreducible and $\pi_1\times\pi_2=\pi(C_1*C_2)$.
In particular, $\pi_1\times\pi_2\times\pi_3$ is \SI.
Thus,
\begin{multline*}
\soc(\pi_1\times\pi_2\times\pi_3)=\soc(\pi_1\times\soc(\pi_2\times\pi_3))=\soc(\pi_1\times\pi(C_2*C_3))\\=
\pi(C_1*(C_2*C_3))=\pi((C_1*C_2)*C_3).
\end{multline*}
Here we used the fact that $C_2$ and $C_1$ are good in the second and third equalities, respectively,
as well as \cite[Corollary 7.4]{2103.12027} for the last one.
In a similar vein we show that
\[
\soc(\pi_3\times\pi_1\times\pi_2)=\pi(C_3*(C_1*C_2)).
\]
It follows that $C_1*C_2$ is good as required.
\end{proof}

For convenience, for $C\in\Irrcomp(\bfd)$ we
set $\max C=\max\{i\mid\bfd(i)\ne0\}$ with the convention that $\max C=0$ if $C=0$.
Similarly define $\min C$.

To go further, we prove the following general reduction step.

\begin{lemma} \label{lem: reduct}
Let $C\in\Irrcomp$.
Assume that $\pi(C*D)$ is a subrepresentation of $\pi(C)\times\pi(D)$
for every $D\in\Irrcomp$ such that $\max D\le\max C$, or alternatively, for every $D\in\Irrcomp$ such that $\min D\le\min C$.
Then, $\pi(C*D)$ is a subrepresentation of $\pi(C)\times\pi(D)$ for every $D\in\Irrcomp$.
\end{lemma}

\begin{proof}
We consider the first case.
We argue by induction on $t:=\max D$.
The statement is trivial if $D=0$.
For the induction step, we may assume of course that $t>\max C$, otherwise there is nothing to prove.
Let $D=\prm_{Q^{\op}}(\n)$ be the $Q^{\op}$-parameterization of $D$ by a multisegment.
Decompose $\n$ as $\n_1+\n_2$ (possibly with $\n_1=0$) where $\n_2\ne0$ consists of the segments $\Delta$ in $\n$ (counted with their multiplicities)
such that $e(\Delta)=t$. Let $D_i=\prm_{Q^{\op}}(\n_i)$, $i=1,2$.
Note that $\max D_1<t$. Then, $D_2$ is rigid and by Lemma \ref{lem: notprec1} and \eqref{eq: notprecPP}, $D=D_1*D_2$ and $\hom_\Pi(D_2,C)=0$.
Therefore, by \cite[Corollary 7.4]{2103.12027}
\[
C*D=(C*D_1)*D_2.
\]
Thus,
\[
\pi(C*D)=\pi((C*D_1)*D_2).
\]
Since $\max(C*D_1)<t$ we may apply Lemma \ref{lem: notprec1} to $C*D_1$ and $D_2$ to obtain
\[
\pi((C*D_1)*D_2)=\soc(\pi(C*D_1)\times\pi(D_2)).
\]
By induction hypothesis,
\[
\pi(C*D_1)\hookrightarrow\pi(C)\times\pi(D_1).
\]
Hence,
\[
\soc(\pi(C*D_1)\times\pi(D_2))\hookrightarrow\soc(\pi(C)\times\pi(D_1)\times\pi(D_2)).
\]
However, by \cite[Lemma 9]{MR4169050}
\[
\soc(\pi(C)\times\pi(D_1)\times\pi(D_2))=\soc(\pi(C)\times\pi(D)).
\]
All in all, we obtain
\[
\pi(C*D)\hookrightarrow\soc(\pi(C)\times\pi(D))
\]
as required.

A similar argument (using $\prm_Q$) applies to the case where $\min D\le\min C$.
\end{proof}

From Lemmas \ref{lem: reduct} and \ref{lem: notprec1} and Remark \ref{rem: symm} we conclude
\begin{corollary} \label{cor: segood}
Let $\Delta$ be a segment. Then, $\prm_Q(\Delta)$ and $\prm_{Q^{\op}}(\Delta)$ are good.
\end{corollary}

\begin{remark} \label{rem: recipe1}
The proof of Lemma \ref{lem: reduct}, together with Lemma \ref{lem: notprec1}, gives a combinatorial recipe for the computation
of $\m*\n$ in the case where $\m$ or $\n$ is a single segment.
Note that the recipe involves the M\oe glin-Waldspurger involution since we switch between $\prm_Q$ and $\prm_{Q^{\op}}$.
\end{remark}

\section{Regular multisegments} \label{sec: regms}

In this section we recall the results of \cite{MR3866895} for regular multisegments
and state our main theoretical result towards Conjecture \ref{conj: subconj}.

We start with a special case.
Consider the quiver $Q=A_{2n-1}$ and graded dimension $\bfd=(1,2,\dots,n,n-1,\dots,1)$
\[
\overset1{\bullet}\rightarrow\overset2{\bullet}\rightarrow\dots\rightarrow\overset{n}{\bullet}\rightarrow
\overset{n-1}{\bullet}\rightarrow\dots\rightarrow\overset1{\bullet}
\]
Fix $V=\oplus_{i=1}^{2n-1}V_i$ with $\dim V_i=\bfd(i)=\min(i,2n-i)$ (so that $\dim V=n^2$).
As in \cite[\S8]{MR1458969}, consider the following open, $G_V$-stable subset of $R_Q(V)$
\[
R_Q^\flat(V)=\{T_+\in R_Q(V)\mid T_+\rest_{V_i}\text{ is injective }\forall i<n\text{ and surjective }\forall i\ge n\}.
\]
For any $T_+\in R_Q^\flat(V)$ we assign the following two complete flags of $V_n$
\begin{gather*}
0\subsetneq T_+^{n-1}(V_1)\subsetneq T_+^{n-2}(V_2)\subsetneq\dots\subsetneq T_+(V_{n-1})\subsetneq V_n,\\
0\subsetneq\Ker(T_+\rest_{V_n})\subsetneq\Ker(T_+^2\rest_{V_n})\subsetneq\dots\subsetneq\Ker(T_+^{n-1}\rest_{V_n})\subsetneq V_n.
\end{gather*}
The resulting map $R_Q^\flat(V)\rightarrow X\times X$, where $X$ is the (complete) flag variety of $\GL_n$,
is a principal $\prod_{i\ne n}\GL(V_i)$-bundle.
Hence, we get an isomorphism of $\GL_n$-varieties
\[
R_Q^{\flat}(V)/\prod_{i\ne n}\GL(V_i)\simeq X\times X.
\]

Thus, the $G_V$-orbits in $R_Q^{\flat}(V)$ correspond to the $\GL_n$-orbits in $X\times X$,
which are parameterized by the symmetric group $S_n$.

If $Y_w$, $w\in S_n$ is a $\GL_n$-orbit in $X\times X$ (i.e., a Bruhat cell), then
\begin{equation} \label{eq: Cw}
C_w=\prm_Q([1,w(1)+n-1]+\dots+[n,w(n)+n-1])\in\Irrcomp(\bfd)
\end{equation}
is the closure of the conormal bundle of the $G_V$-orbit in $R_Q^{\flat}(V)$ corresponding to $Y_w$.

Denote by $X_w$ the closure of $Y_w$ (i.e., the Schubert variety).

For example, $X_e=Y_e=\Delta X$ (diagonal embedding) for the identity permutation,
while for the longest element, $Y_{w_0}$ is open and $X_{w_0}=X\times X$.

The following result affirms a special case of Conjecture \ref{conj: GSR}.

\begin{theorem}[\cite{MR3866895}] \label{thm: LM}
The following conditions on $w\in S_n$ are equivalent.\footnote{The equivalence of the first two conditions is easy.}
\begin{enumerate}
\item $C_w$ is rigid.
\item \label{part: open} The conormal bundle of $Y_w\subset X$ admits an open $\GL_n$-orbit.
\item \label{part: smth} $X_{w_0w}$ is smooth.
\item $\pi(C_w)\times\pi(C_w)$ is irreducible.
\end{enumerate}
\end{theorem}
We also recall that $X_{w_0w}$ is smooth $\iff$ $X_{w_0w}$ is rationally smooth $\iff$
$w$ is $1324$ and $2143$ avoiding \cite{MR788771, MR1051089}.

Note that Example \ref{ex: nonrigid} is not covered by Theorem \ref{thm: LM}, but we
will shortly discuss the more general results of \cite{MR3866895}, which
cover this example.

In passing, we mention that
the equivalence of the purely geometric conditions \ref{part: open} and \ref{part: smth} (in the more general context)
led Anton Mellit to make the following conjecture, for which there is overwhelming computational support.
\begin{conjecture}[Mellit] \label{conj: Mellit}
Let $x,w\in S_n$ with $Y_w\subset X_x$ (i.e., $w\le x$ in the Bruhat order).
Suppose that $X_x$ is smooth. Then, the following conditions are equivalent.
\begin{enumerate}
\item The conormal bundle of $Y_w\subset X_x$ admits an open $\GL_n$-orbit.
\item The smooth locus of $X_{w_0w}$ contains $Y_{w_0x}$.
\end{enumerate}
\end{conjecture}
In the case where $x$ is $231$-avoiding (which implies that $X_x$ is smooth), this conjecture (along with a representation-theoretic criterion) was proved in \cite{MR3866895},
albeit indirectly.

\begin{remark}
In general, for any two parabolic subgroups $P$ and $Q$ of $\GL_n$, one can realize in a similar way the $\GL_n$-action on $P\bs\GL_n\times Q\bs\GL_n$
in the geometry of a $G_V$-action on an open subset of $R_Q(V)$ for a suitable graded vector space $V$.
Unfortunately, the naive generalization of Theorem \ref{thm: LM} to this context is not true --
nor do we have a conjectural replacement for the smoothness condition.
\end{remark}

For any segment $\Delta=[a,b]$ we write $b(\Delta)=a$ and $e(\Delta)=b$.

The multisegments considered above satisfy the following property (cf.\ \eqref{eq: Cw}):
writing $\m=\Delta_1+\dots+\Delta_k$, we have
$b(\Delta_i)\ne b(\Delta_j)$ and $e(\Delta_i)\ne e(\Delta_j)$ for all $i\ne j$.
Such multisegments are called regular.

Recall the setup of \cite[\S6]{MR3866895}.
We say that a (necessarily regular) multisegment is of type $4231$ if it can be written as $\m=\Delta_1+\dots+\Delta_k$ with $k\ge4$
such that the following conditions are satisfied.\footnote{Note that in \cite[Definition 6.10]{MR3866895} a different ordering of the $\Delta_i$'s
was used, justifying the terminology.}
\begin{enumerate}
\item $\Delta_i\prec\Delta_{i-1}$, $i=3,\dots,k$.
\item $b(\Delta_k)<b(\Delta_1)<b(\Delta_{k-1})$.
\item $e(\Delta_3)<e(\Delta_1)<e(\Delta_2)$.
\end{enumerate}
Similarly, a (necessarily regular) multisegment is of type $3412$ if it can be written as $\m=\Delta_1+\dots+\Delta_k$ with $k\ge4$
such that the following conditions are satisfied.
\begin{enumerate}
\item $\Delta_i\prec\Delta_{i-1}$, $i=4,\dots,k$ and $\Delta_2\prec\Delta_1$.
\item $b(\Delta_2)<b(\Delta_k)<b(\Delta_1)<b(\Delta_{k-1})$.
\item $e(\Delta_4)<e(\Delta_2)<e(\Delta_3)<e(\Delta_1)$.
\end{enumerate}

We say that a regular multisegment is balanced if it does not admit any multisegment of type $4231$ or type $3412$ as a submultisegment.
The main result of \cite{MR3866895}, for which Theorem \ref{thm: LM} is a special case, is the following.
\begin{theorem}[\cite{MR3866895}] \label{thm: mainLM}
Suppose that $\m$ is a regular multisegment.
Then, the following conditions are equivalent.
\begin{enumerate}
\item $\m$ is balanced.
\item $Z(\m)$ is $\square$-irreducible.
\item $\prm_Q(\m)$ is rigid.
\end{enumerate}
\end{theorem}

We can now state our main result towards Conjecture \ref{conj: subconj}.
\begin{theorem} \label{thm: cor}
Suppose that $C=\prm_Q(\m)$ where $\m$ is a balanced (regular) multisegment.
Then, $C$ is good.
In particular, by Lemma \ref{lem: prod}, every rigid laminated module (see \S\ref{sec: lad}) is good.
\end{theorem}

Theorem \ref{thm: cor} considerably extends Theorem 1 of \cite{MR4169050}.
(The latter applies only to rigid laminated modules and proves a weak version of Conjecture \ref{conj: subconj}.)
We will prove Theorem \ref{thm: cor} in \S\ref{sec: endofproof} below
by combining ideas from \cite{MR4169050} and \cite{MR3866895} and pushing them further.

\subsection*{Addendum: primality for regular multisegments\footnote{The results of this subsection will not be used in the rest of the paper.}}

As a complement, we give a (surprisingly?)\ simple necessary and sufficient condition for the property
that $Z(\m)$ is a prime (i.e., irreducible) element in $\BZ$, in the case where $\m$ is regular.

Note that if $Z(\m)$ is prime, then we cannot write non-trivially $Z(\m)=Z(\m_1)\times Z(\m_2)$
(where necessarily $\m=\m_1+\m_2$).
In general, it is not clear whether the converse holds.

We first need an elementary lemma. For this section only, let $A$ be an integral domain.
Let $R=A[x_i,i\in I]$ be the (commutative) ring of polynomials over $A$ in the variables $x_i$, $i\in I$
for some finite nonempty set $I$.
For any subset $J\subset I$ let $x_J\in R$ be the monomial $\prod_{i\in J}x_i$.

\begin{lemma} \label{lem: irredet}
Consider a polynomial $f\in R$ of the form $f=\sum_{J\subset I}c_Jx_J$, $c_j\in A$.

Assume that for any $i\in I$, there exists $J\subset I$ such that $i\in J$ and $c_J\ne0$.

Consider the graph $\Gamma$ whose vertex set is $I$ and we connect $i$ and $j$
if there is no $J\subset I$ such that $i,j\in J$ and $c_J\ne0$.
Assume that $\Gamma$ is connected.

Finally, assume that $\sum_{J\subset I} Ac_J=A$.

Then, $f$ is irreducible in $R$.
\end{lemma}

\begin{proof}
By assumption, the degree of $x_i$ in $f$ is $1$ for every $i\in I$.

Suppose that $f=gh$. Then, for each $i$, precisely one of $g$ and $h$ is independent of $x_i$
(i.e., belongs to $A[x_j,j\ne i]$).

We claim that if $g$ is independent of $x_i$, then $g$ is also independent of $x_j$
for any edge $(i,j)$ in $\Gamma$.

Indeed, write $h=a+bx_i$ where $a,b\in A[x_k,k\ne i]$ and $b\ne0$.
Assume on the contrary that $g$ is dependent on $x_j$.
Then, $h$ is independent of $x_j$ and therefore $a,b\in A[x_k,k\ne i,j]$.
Write $g=c+dx_j$ with $c,d\in A[x_k,k\ne i,j]$ and $d\ne0$.
Then, $f=ac+bcx_i+adx_j+bdx_ix_j$
and therefore $x_ix_j$ appears non-trivially in $f$, in contradiction to the assumption that $(i,j)$ is an edge.

By our assumption on $\Gamma$, it follows that either $g$ or $h$ is a constant, which
must be a unit by our last assumption.
\end{proof}

Let now $\m$ be a multisegment.
We say that $\m$ is \emph{split} if we can write $\m=\m_1+\m_2$ nontrivially, where
for every two segments $\Delta_1$ in $\m_1$ and $\Delta_2$ in $\m_2$ we have
$\Delta_1\not\prec\Delta_2$ and $\Delta_2\not\prec\Delta_1$ (i.e., $Z(\Delta_1)\times Z(\Delta_2)$ is irreducible).

\begin{proposition} \label{prop: regprime}
The following conditions are equivalent for a regular multisegment $\m$.
\begin{enumerate}
\item \label{part: prime} $Z(\m)$ is not a prime element of $\BZ$.
\item \label{part: decomp} $Z(\m)=Z(\m_1)\times Z(\m_2)$ for some nontrivial decomposition $\m=\m_1+\m_2$.
\item \label{part: split} $\m$ is split.
\end{enumerate}
\end{proposition}

\begin{proof}
Clearly, \ref{part: split}$\implies$\ref{part: decomp}$\implies$\ref{part: prime}.
(These implications hold without the assumption that $\m$ is regular.)

It remains to prove that if $\m$ is not split (but regular), then $Z(\m)$ is prime in $\BZ$.

Write $\m=\Delta_1+\dots+\Delta_k$ and
\begin{equation} \label{eq: decs}
\std(\m)=\sum_{\n\in S}c'_\n Z(\n),
\end{equation}
(see \eqref{eq: std}) where $c'_\n$ are positive integers for all $\n$. We have $c'_\m=1$, and the set $S$
was described explicitly by Zelevinsky \cite[\S7]{MR584084}.
Denote by $I'$ the set of all segments that occur in a multisegment in $S$.
In particular, $I'\supset\{\Delta_1,\dots,\Delta_k\}$.
The following facts follow easily from Zelevinsky's description.
\begin{enumerate}
\item Every $\n\in S$ is regular.
\item For every $\Delta\in I'$ we have $b(\Delta)\in\{b(\Delta_1),\dots,b(\Delta_k)\}$ and
$e(\Delta)\in\{e(\Delta_1),\dots,e(\Delta_k)\}$.
\item Suppose that $\Delta_r\prec\Delta_s$, but there does not exist $t$ such that
$\Delta_r\prec\Delta_t$ and $\Delta_t\prec\Delta_r$. (For brevity, we will write $\Delta_r\iprec\Delta_s$ in this case.)
Let $\n$ be the multisegment obtained from $\m$ by replacing $\Delta_r$ and $\Delta_s$
by $\Delta_r\cup\Delta_s$ and $\Delta_r\cap\Delta_s$ (or just $\Delta_r\cup\Delta_s$ if $\Delta_r\cap\Delta_s=\emptyset$).
We will say that $\n$ is a neighbor of $\m$.
Then, $\n\in S$, and in particular $\Delta_r\cup\Delta_s\in I'$. Moreover, for any $\m'\in S\setminus\{\m,\n\}$,
the representation $Z(\n)$ does not occur in $\std(\m')$.
\end{enumerate}
Inverting the relation \eqref{eq: decs} (and its analogues for $\std(\n)$, for any $\n\in S$) we get
\begin{equation} \label{eq: decZ}
Z(\m)=\sum_{\n\in S}c_\n\std(\n)
\end{equation}
where $c_\n$ are integers with $c_\m=1$.
The only other pertinent piece of information for us is that $c_\n=-c'_\n\ne0$ for every neighbor $\n$ of $\m$.
(In fact, one can show that $c'_\n=1$ in this case.)\footnote{In general, the $c'_\n$'s (and hence also the $c_\n$'s, up to explicit signs) are given
by the value at $1$ of certain Kazhdan--Lusztig polynomials with respect to the
symmetric group $S_k$ -- see e.g. \cite{MR2320806}.
In particular, $c_\n\ne0$ for all $n\in S$. However, we will not use this fact.}

Let $I$ be the set of segments that occur in a multisegment $\n\in S$ for which $c_\n\ne0$.
Thus, $\{\Delta_1,\dots,\Delta_k\}\subset I\subset I'$ and if $\Delta_r\iprec\Delta_s$, then
$\Delta_r\cup\Delta_s\in I$. (In reality, $I=I'$.)
Consider the free variables $Z(\Delta)$, $\Delta\in I$.
The relation \eqref{eq: decZ} expresses $f=Z(\m)\in\BZ$ as a sum
of monomials in $Z(\Delta)$, $\Delta\in I$ (corresponding to $\std(\n)$) with coefficients $c_\n$.
Consider the graph $\Gamma$ defined in Lemma \ref{lem: irredet}.

Let $\Delta\in I$. If $r$ is such that $b(\Delta)=b(\Delta_r)$,
then $(\Delta,\Delta_r)$ is an edge in $\Gamma$ (since every $\n\in S$ is regular).
Similarly, if $s$ is such that $e(\Delta)=e(\Delta_s)$,
then $(\Delta,\Delta_s)$ is an edge in $\Gamma$.

Suppose that $\Delta_r\iprec\Delta_s$.
Then, $\Delta_r\cup\Delta_s\in I$ and both $(\Delta_r,\Delta_r\cup\Delta_s)$ and $(\Delta_r\cup\Delta_s,\Delta_s)$
are edges in $\Gamma$ (by the previous paragraph).

On the other hand, the fact that $\m$ is not split means that
the graph whose vertices are $\{1,\dots,k\}$ and whose edges are given by $\Delta_r\iprec\Delta_s$, is connected.
It follows that $\Gamma$ is connected.
The primality of $Z(\m)$ in the subring of $\BZ$ generated by $Z(\Delta)$, $\Delta\in I$
(and hence, the primality of $Z(\m)$ in $\BZ$ itself), follows from Lemma \ref{lem: irredet} (for $A=\Z$).
\end{proof}

\begin{remark}
Proposition \ref{prop: regprime} ceases to hold if we only assume that all the segments in $\m$ occur with multiplicity one.
For instance, consider $\m=\m_1+\m_2$ where
\[
\m_1=[1,2]+[2,3]=\Delta_1+\Delta_2,\ \ \m_2=[3,3]=\Delta_3.
\]
Then, $Z(\m)=Z(\m_1)\times Z(\m_2)$ even though $\m$ is not split since $\Delta_1\prec\Delta_2$ and $\Delta_1\prec\Delta_3$.
\end{remark}

\section{Basic representations and irreducible components} \label{sec: Basic}

In this section we prove an analogue of a special case of the Baumann--Kamnitzer--Tingley decomposition \cite{MR3270589}
for the representation theory of $\GL$.

\subsection{} \label{sec: defbasic}
By definition, a \emph{basic} representation is one of the form $\sigma=Z(\Delta)$ or $\sigma=L(\Delta)$ for some segment $\Delta$.\footnote{The discussion of this subsection and the next one is valid for any segment
in the sense of Zelevinsky.}

Let $\sigma $ be a basic representation. \label{sec: basicdef}
Depending on whether $\sigma=Z(\Delta)$ or $\sigma=L(\Delta)$, we say that a basic representation is $\sigma$-saturated
if it is of the form $Z(\Delta')$ (resp., $L(\Delta')$) where $\Delta'\subset\Delta$ and $e(\Delta')=e(\Delta)$ (resp., $b(\Delta')=b(\Delta)$).
We denote by $\Irrsatb$ the set of basic $\sigma$-saturated representations.

Note that in both cases, if $\sigma_1,\dots,\sigma_k\in\Irrsatb$, then $\sigma_1\times\dots\times\sigma_k$ is $\square$-irreducible.
Such a product (possibly with $k=0$) will be called a $\sigma$-saturated representation.
We denote by $\Irrsat$ the set of $\sigma$-saturated representations.
Any $\sigma'\in\Irrsat$ can be written uniquely, up to reordering of the factors, as $\sigma_1\times\dots\times\sigma_k$ where $\sigma_i\in\Irrsatb$.
We will call $k$ the $\sigma$-index of $\sigma'$.

More generally, let $\pi\in\Irr$.
Depending on whether $\sigma=Z(\Delta)$ or $\sigma=L(\Delta)$, we define
the $\sigma$-index of $\pi$, denoted by $\indx_\sigma(\pi)$, to be the multiplicity of $e(\Delta)$ (resp., $b(\Delta)$)
in the cuspidal support of $\pi$ if the latter is contained (as a set) in $\Delta$, and $-\infty$ otherwise.
This definition is consistent with the case of $\sigma$-saturated representations.

For instance, if $\pi=Z(\Delta_1+\dots+\Delta_k)$ with $\Delta_i\subset\Delta$ for all $i$, then the $Z(\Delta)$-index
of $\pi$ is the number of $i$'s such that $e(\Delta_i)=e(\Delta)$.
Similarly, if $\pi=L(\Delta_1+\dots+\Delta_k)$ with $\Delta_i\subset\Delta$ for all $i$, then the $L(\Delta)$-index
of $\pi$ is the number of $i$'s such that $b(\Delta_i)=b(\Delta)$.

In particular, if $\indx_\sigma(\pi)>0$, then there exist $\sigma'\in\Irrsatb$ and $\pi'\in\Irr$ such that
\begin{equation} \label{eq: ind>0}
\pi=\soc(\sigma'\times\pi').
\end{equation}
(Namely, if $\sigma=Z(\Delta)$ and $\pi=Z(\Delta_1+\dots+\Delta_k)$, then we may take $\sigma'=Z(\Delta_i)$
for any $i$ such that $e(\Delta_i)=e(\Delta)$; if $\sigma=L(\Delta)$ and $\pi=L(\Delta_1+\dots+\Delta_k)$, then we may take $\sigma'=L(\Delta_i)$ for any $i$ such that $b(\Delta_i)=b(\Delta)$.)

It is also clear that if  $\pi_1, \pi_2\in\Irr$, then for any irreducible subquotient $\tau$ of $\pi_1\times\pi_2$ we have
\begin{equation} \label{eq: addind}
\indx_\sigma(\tau)=\indx_\sigma(\pi_1)+\indx_\sigma(\pi_2).
\end{equation}

Denote by\footnote{Here $\boxtimes$ denotes Deligne's tensor product of categories -- see e.g.\ \cite[\S1.11]{MR3242743}}
\[
J:\Reps\rightarrow\Reps\boxtimes\Reps
\]
the ``total'' Jacquet module, i.e., $J$ is the direct sum over $m\ge0$ of the functors
\begin{align*}
\Reps(\GL_m(F))\rightarrow&\oplus_{m_1,m_2\ge0:m_1+m_2=m}\Reps(\GL_{m_1}(F)\times\GL_{m_2}(F))\\=&
\oplus_{m_1,m_2\ge0:m_1+m_2=m}\Reps(\GL_{m_1}(F))\boxtimes\Reps(\GL_{m_2}(F))
\end{align*}
where for each summand we take the Jacquet module with respect to the standard parabolic subgroup of type
$(m_1,m_2)$ of $\GL_m(F)$.

Note that if $\sigma'\in\Irrsatb$, then for every irreducible subquotient $\pi_1\otimes\pi_2$ of $J(\sigma')$
other than $\sigma'\otimes\one$ we have $\indx_\sigma(\pi_1)=0$.
More generally, it follows from the geometric lemma of Bernstein and Zelevinsky that if $\sigma'\in\Irrsat$, then for every irreducible subquotient $\pi_1\otimes\pi_2$ of $J(\sigma')$
other than $\sigma'\otimes\one$ we have
\begin{equation}\label{eq: satmin2}
\indx_\sigma(\pi_1)<\indx_\sigma(\sigma').
\end{equation}

\subsection{} \label{sec: basic2}
For any $\pi\in\Reps$ (not necessarily irreducible) we define the hereditary $\sigma$-index by
\[
\hind_\sigma(\pi)=\max\{\indx_\sigma(\pi_1)\mid\pi_1\otimes\pi_2\text{ is an irreducible subquotient of }J(\pi)\},
\]
interpreted as $0$ if $\pi=0$.
We say that $\pi$ is $\sigma$-reduced if its hereditary $\sigma$-index is $0$.
We denote by $\Repsred\subset\Reps$ the subcategory of $\sigma$-reduced representations and by $\Irred\subset\Irr$ its irreducible objects.

Clearly, $\hind_\sigma(\pi)\ge\max(0,\indx_\sigma(\pi))$ for any $\pi\in\Irr$, and
$\hind_\sigma(\pi_1)\le\hind_\sigma(\pi_2)$ if $\pi_1$ is a subquotient of $\pi_2$.

It follows from \eqref{eq: addind} and the geometric lemma that for any $\pi_1,\pi_2\in\Reps$ we have
\begin{equation} \label{eq: hindadd}
\hind_\sigma(\pi_1\times\pi_2)=\hind_\sigma(\pi_1)+\hind_\sigma(\pi_2).
\end{equation}
In particular,
\begin{equation} \label{eq: prodred}
\text{if $\pi_1$ and $\pi_2$ are $\sigma$-reduced, then $\pi_1\times\pi_2$ is also $\sigma$-reduced.}
\end{equation}

Also, it follows from \eqref{eq: satmin2} that if $\sigma'\in\Irrsat$,
then for every irreducible subquotient $\pi_1\otimes\pi_2$ of $J(\sigma')$ other than $\sigma'\otimes\one$ we have
\begin{equation}\label{eq: satmin}
\hind_\sigma(\pi_1)<\indx_\sigma(\sigma').
\end{equation}
In particular,
\begin{equation} \label{eq: hindsat}
\hind_\sigma(\sigma')=\indx_\sigma(\sigma').
\end{equation}

We say that two representations $\pi_1, \pi_2\in\Reps$ have \emph{disjoint supports} if for every irreducible subquotient
$\tau_i$ of $\pi_i$, $i=1,2$ the cuspidal support of $\tau_1$ and $\tau_2$ are distinct as multisets.
We use similar terminology for representations of $\GL\times\GL$.

\label{sec: fork}

For any $\pi_1,\pi_2\in\Reps$ we denote by $\KJ_{\pi_1,\pi_2}$ the kernel of the canonical surjection
\[
J(\pi_1\times\pi_2)\twoheadrightarrow\pi_1\otimes\pi_2.
\]
Following \cite[\S7.2]{MR4169050} we write $\pi_1\sprt\pi_2$ if $\KJ_{\pi_1,\pi_2}$ and $\pi_1\otimes\pi_2$ have disjoint supports (as representations of $\GL\times\GL$).

By the geometric lemma, \eqref{eq: satmin} and \cite[Corollary 5]{MR4169050} we get

\begin{lemma} \label{lem: prodSI}
Let $\pi_1\in\Irrsat$ and $\pi_2\in\Repsred$.
Then, for any irreducible subquotient $\tau_1\otimes\tau_2$ of $\KJ_{\pi_1,\pi_2}$ we have
\[
\hind_\sigma(\tau_1)<\indx_\sigma(\pi_1)\text{ or }\deg\tau_1>\deg\pi_1.
\]
In particular, $\pi_1\sprt\pi_2$. Hence, if $\pi_2$ is SI, then $\pi_1\times\pi_2$ is SI.
\end{lemma}

Fix a basic representation $\sigma$.

\begin{lemma} \label{lem: red}
The following conditions are equivalent for $\pi\in\Reps$.
\begin{enumerate}
\item \label{part: hind>0} $\hind_\sigma(\pi)>0$.
\item \label{part: Jind>0} $J(\pi)$ admits an irreducible subquotient of the form $\sigma'\otimes\pi'$ with $\indx_\sigma(\sigma')>0$.
\item \label{part: Jbasicsq} $J(\pi)$ admits an irreducible subquotient of the form $\sigma'\otimes\pi'$ where $\sigma'\in\Irrsatb$.
\item \label{part: Kqbasic} $J(\pi)$ admits an irreducible quotient of the form $\sigma'\otimes\pi'$ where $\sigma'\in\Irrsatb$.
\item \label{part: embedbasic} There exists a non-zero homomorphism $\pi\rightarrow\sigma'\times\pi'$ for some $\sigma'\in\Irrsatb$ and $\pi'\in\Irr$.
\end{enumerate}
In particular, if $\pi\in\Irr$, then $\pi$ is $\sigma$-reduced if and only if
\[
\pi\not\hookrightarrow\sigma'\times\pi'
\]
for every $\sigma'\in\Irrsatb$ and $\pi'\in\Irr$.
\end{lemma}

\begin{proof}
\ref{part: hind>0}$\iff$\ref{part: Jind>0} by definition of $\hind$.

\ref{part: Kqbasic}$\implies$\ref{part: Jbasicsq}$\implies$\ref{part: Jind>0} is obvious.

Conditions \ref{part: Kqbasic} and \ref{part: embedbasic} are equivalent by Frobenius reciprocity.

Finally, suppose that $\tau\otimes\omega$ is an irreducible subquotient of $J(\pi)$ such that $\indx_\sigma(\tau)>0$.
Since $\indx_\sigma(\tau)$ depends only on the cuspidal support of $\tau$, we may assume without loss of generality that $\tau\otimes\omega$ is a quotient of $J(\pi)$.
Thus, by Frobenius reciprocity there exists a non-zero homomorphism $\pi\rightarrow\tau\times\omega$.
Hence, by \eqref{eq: ind>0}, there exists a non-zero homomorphism $\pi\rightarrow\sigma'\times\tau'\times\omega$ for some
$\sigma'\in\Irrsatb$ and $\tau'\in\Irr$.
Hence, there exists a non-zero homomorphism $\pi\rightarrow\sigma'\times\pi'$ for some irreducible subquotient $\pi'$ of $\tau'\times\omega$.
Thus, \ref{part: Jind>0}$\implies$\ref{part: embedbasic}.
\end{proof}

\begin{proposition} \label{prop: decomDelta}
Let $\pi$ be an irreducible representation of $\GL$. Then, there exist $\pi_1\in\Irrsat$ and $\pi_2\in\Irred$ such that
\begin{equation} \label{eq: sigmadecomp}
\pi=\soc(\pi_1\times\pi_2).
\end{equation}
Moreover, $\pi_1$ and $\pi_2$ are uniquely determined by these conditions.
More precisely,
\begin{enumerate}
\item \label{part: hind2} $\hind_\sigma(\pi)=\hind_\sigma(\pi_1)=\indx_\sigma(\pi_1)$.
\item \label{part: uniqueJ} $\pi_1\otimes\pi_2$ occurs with multiplicity one in the Jordan--H\"older sequence of $J(\pi)$.
\item \label{part: onlykd} If $\pi'_1\otimes\pi'_2$ is an irreducible subquotient of $J(\pi)$ such that $\hind_\sigma(\pi'_1)=\hind_\sigma(\pi)$,
then either $\pi'_1\otimes\pi'_2=\pi_1\otimes\pi_2$ or $\deg\pi'_1>\deg\pi_1$.
\end{enumerate}
Thus, $\pi_1\otimes\pi_2$ is the unique irreducible subquotient of $J(\pi)$ with $\deg\pi_1$ minimal such that
$\hind_\sigma(\pi_1)=\hind_\sigma(\pi)$.
\end{proposition}

We will refer to \eqref{eq: sigmadecomp} as the \emph{$\sigma$-decomposition} of $\pi$.

\begin{proof}
The existence follows from Lemma \ref{lem: red} by induction on the degree of $\pi$.
Suppose that \eqref{eq: sigmadecomp} holds with $\pi_1\in\Irrsat$ and $\pi_2\in\Irred$.
Then, $\pi_1\otimes\pi_2$ occurs as a quotient of $J(\pi)$ and hence,
\begin{gather*}
\indx_\sigma(\pi_1)\le\hind_\sigma(\pi)\le
\hind_\sigma(\pi_1\times\pi_2)\stackrel{\eqref{eq: hindadd}}=\hind_\sigma(\pi_1)+\hind_\sigma(\pi_2)\\=
\hind_\sigma(\pi_1)\stackrel{\eqref{eq: hindsat}}=\indx_\sigma(\pi_1),
\end{gather*}
which implies the first part. The other parts follow from Lemma \ref{lem: prodSI}.
\end{proof}

\begin{lemma}
Let $\pi_1$, $\pi_2$ be irreducible representations with $\sigma$-decompositions
\[
\pi_i=\soc(\sigma_i\times\tau_i),\ i=1,2.
\]
If $\pi_1\times\pi_2$ is irreducible, then its $\sigma$-decomposition is
\[
\pi_1\times\pi_2=\soc(\sigma'\times\tau')\text{ where }\sigma'=\sigma_1\times\sigma_2\text{ and }\tau'=\tau_1\times\tau_2.
\]
In particular, $\tau_1\times\tau_2$ is irreducible.

Hence, if $\pi$ is $\square$-irreducible with $\sigma$-decomposition $\pi=\soc(\pi_1\times\pi_2)$,
then $\pi_2$ is $\square$-irreducible.
\end{lemma}

\begin{proof}
Clearly, $\sigma'\in\Irrsat$, $\tau'\in\Repsred$, $\sigma'\otimes\tau'$ is a subquotient of $J(\pi_1\times\pi_2)$ and
$\indx_\sigma(\sigma')=\hind_\sigma(\pi_1\times\pi_2)$.
Moreover, by Lemma \ref{lem: prodSI}, $J(\sigma'\times\tau')$, and hence also $J(\pi_1\times\pi_2)$, does not admit an irreducible subquotient
$\omega_1\otimes\omega_2$ such that $\hind_\sigma(\omega_1)=\hind_\sigma(\pi_1\times\pi_2)$ and
$\deg\omega_1<\deg\sigma'$. The Lemma therefore follows from Proposition \ref{prop: decomDelta}.
\end{proof}

\subsection{} \label{sec: basicomp}

We turn to the analogous concepts for irreducible components.

First, we recall a general fact.
Let $R$ be any ring and $z$ a module over $R$.
We say that an $R$-module $x$ is $z$-saturated (resp., $z$-reduced) if $\Hom_R(x',z)\ne0$ for every non-trivial submodule $x'$ of $x$
(resp., $\Hom_R(x,z)=0$).
The class of $z$-saturated modules contains $z$ itself, and is closed under arbitrary direct products, passing to submodules, and taking extensions.
The class of $z$-reduced modules is closed under quotients, arbitrary direct sums, and extensions.
Note that if $x_1$ is $z$-saturated and $x_2$ is $z$-reduced, then $\Hom_R(x_2,x_1)=0$.

Any module $x$ admits a canonical short exact sequence
\begin{equation} \label{eq: sxz}
0\rightarrow x[z]\rightarrow x\rightarrow x\{z\}\rightarrow0
\end{equation}
where $x[z]$ is $z$-reduced and $x\{z\}$ is $z$-saturated.
Namely, $x[z]$ is the maximal $z$-reduced submodule of $x$.
Moreover, if $y$ is any $z$-saturated module, then any morphism $x\rightarrow y$ factors through
$x\twoheadrightarrow x\{z\}$.

In other words, the (full) subcategories of $z$-reduced and $z$-saturated objects form a torsion pair in the category of $R$-modules
in the sense of \cite[\S3.1]{MR3270589}.

We will apply this discussion to the case of (finite-dimensional) modules over the preprojective algebra $\Pi$.

Fix a basic representation $\sigma$ and let $C_\sigma$ be the corresponding irreducible component.
(We will call $C_\sigma$ a \emph{basic irreducible component}.)
Note that $C_\sigma$ is rigid. Let $x_\sigma\in C_\sigma$ be rigid.
Note that the submodules of $x_\sigma$ are rigid and correspond to the basic $\sigma$-saturated representations of $\GL$
(see \S\ref{sec: basicdef}).
Moreover, $\Ext_\Pi^1(x_1,x_2)=0$ for any submodules $x_1,x_2$ of $x_\sigma$.
An irreducible component containing a submodule of $x_\sigma$ will be called a \emph{subcomponent} of $C_\sigma$.

\begin{lemma} \label{lem: satcomp}
The $x_\sigma$-saturated modules are the direct sums of submodules of $x_\sigma$.
They are all rigid.
\end{lemma}

\begin{proof}
Clearly, direct sums of submodules of $x_\sigma$ are $x_\sigma$-saturated, and by the above, they are all rigid.
Conversely, suppose that $x$ is $x_\sigma$-saturated, say of graded dimension $\bfd=(d_1,\dots,d_n)$.
Assume for concreteness that $\sigma=Z(\Delta)$. It is clear that $d_i=0$ for all $i\notin\Delta$, otherwise
we would have $\Hom_\Pi(x',x_\sigma)=0$ for the submodule $x'$ generated by a homogeneous vector of degree $i$.
Suppose that $\pi_Q(x)=M_Q(\Delta_1+\dots+\Delta_k)$. By the above, $\Delta_i\subset\Delta$ for all $i$.
We show that $e(\Delta_i)=e(\Delta)$ for all $i$. This will finish the proof because then,
$x=x_{Z(\Delta_1)}\oplus\dots\oplus x_{Z(\Delta_k)}$.

Assume on the contrary that $e(\Delta_i)<e(\Delta)$ for some $i$, and choose $i$ for which $e(\Delta_i)$ is minimal.
Then, $\Delta_j\not\prec\Delta_i$ for all $j$ and therefore, $x$ admits $x_{Z(\Delta_i)}$ as a submodule,
whereas $\Hom_\Pi(x_{Z(\Delta_i)},x_\sigma)=0$. We get a contradiction.

The argument for the case $\sigma=L(\Delta)$ is similar.
\end{proof}

We say that an irreducible component is $\sigma$-saturated if it contains an $x_\sigma$-saturated module.
Thus, $C$ is $\sigma$-saturated if and only if $C=C_1*\dots*C_k=C_1\oplus\dots\oplus C_k$
for some subcomponents $C_1,\dots,C_k$ of $C_\sigma$.
The $\sigma$-saturated irreducible components are rigid and correspond to the $\sigma$-saturated representations of $\GL$.
By Corollary \ref{cor: segood} and Lemma \ref{lem: prod},
all $\sigma$-saturated irreducible components are good.

Finally, we say that an irreducible component $C$ is $\sigma$-reduced if it contains an $x_\sigma$-reduced module, i.e., if $\hom_\Pi(C,C_\sigma)=0$.

\begin{lemma} \label{lem: sigfctr}
Every irreducible component $C$ can be written uniquely as $C_1*C_2$ where $C_1$ is $\sigma$-saturated and $C_2$ is $\sigma$-reduced.
\end{lemma}
We call this the $\sigma$-decomposition of $C$.

\begin{proof}
Suppose that $C\subset\Lambda_{\bfd}$ and let $d$ be the total dimension of $\bfd$.
For any module $y$, the set
\[
\Lambda_{\bfd}(y):=\{x\in\Lambda_{\bfd}\mid\exists\text{ an epimorphism }x\twoheadrightarrow y\}
\]
is constructible.
Note that for any $\sigma$-saturated module $x_1$ of dimension $d_1$ say, we have
\[
\{x\in\Lambda_{\bfd}\mid x\{x_\sigma\}=x_1\}=\Lambda_{\bfd}(x_1)\setminus\cup_y\Lambda_{\bfd}(y)
\]
where $y$ ranges over all $x_\sigma$-saturated modules with $d_1<\dim y\le d$
for which there exists an epimorphism $y\twoheadrightarrow x_1$.
Since up to isomorphism, there are only finitely many $x_\sigma$-saturated modules of a given dimension,
it follows that there exists an $x_\sigma$-saturated module $x_1$ such that $x\{x_\sigma\}=x_1$
for all $x$ in an open, nonempty subset $C'$ of $C$.

Let $C_1$ be the irreducible component containing $x_1$.
Since every morphism $x\rightarrow x\{x_\sigma\}$ factors through the canonical projection
$x\twoheadrightarrow x\{x_\sigma\}$,
we infer from \cite[Proposition 5.1]{2103.12027} that $C=C_1*C_2$ where $C_2\in\Irrcomp$ is the closure of $\{x[x_\sigma]\mid x\in S\}$
for a suitable open subset $S\ne\emptyset$ of $C'$.
It is clear that $C_2$ is $\sigma$-reduced.

Conversely, suppose that $C=C_1'*C'_2$ where $C_1'$ is $\sigma$-saturated and $C_2'$ is $\sigma$-reduced.
Let $x_1'$ be a rigid element of $C_1'$ and
\[
C_2''=\{x\in C_2'\mid\Hom_\Pi(x,x_\sigma)=0\},
\]
an open nonempty subset of $C_2'$. Then, any short exact sequence
\[
0\rightarrow x_2\rightarrow x\rightarrow x_1'\rightarrow 0
\]
with $x_2\in C_2''$ is isomorphic to \eqref{eq: sxz} (with $z=x_\sigma$). In particular, if also $x\in C'$, then
necessarily $x_1'=x_1$ and $x_2\in C_2$. The uniqueness follows.
\end{proof}

\begin{remark}
In fact, Lemma \ref{lem: sigfctr} is a very special case of the results of \cite{MR3270589}.
(Cf.\ [ibid., \S5.5].)
More precisely, in the notation of [ibid.] the set of $\sigma$-saturated (resp., $\sigma$-reduced) irreducible components is
$\Torf^w$ (resp., $\Tors^w$) where $w$ is the product of the simple reflections pertaining to the elements of $\Delta$, ordered by $Q$ or $Q^{\op}$,
depending on whether $\sigma=Z(\Delta)$ or $L(\Delta)$.
\end{remark}

The $\sigma$-decomposition is compatible with that of Proposition \ref{prop: decomDelta}. More precisely,
\begin{corollary}
Suppose that $C=C_1*C_2$ is the $\sigma$-decomposition of $C$. Then,
\begin{equation} \label{eq: pi12}
\pi(C)=\soc(\pi(C_1)\times\pi(C_2))
\end{equation}
is the $\sigma$-decomposition of $\pi(C)$.
In particular, $C$ is $\sigma$-reduced if and only if $\pi(C)$ is $\sigma$-reduced.
\end{corollary}

\begin{proof}
Indeed, since $C_1$ is good, \eqref{eq: pi12} holds.
Clearly, $\pi(C_1)$ is $\sigma$-saturated.
It remains to show that if $C$ is $\sigma$-reduced, then $\pi(C)$ is $\sigma$-reduced.
Suppose on the contrary that $\pi(C)$ is not $\sigma$-reduced.
Then, by Lemma \ref{lem: red}, $\pi(C)=\soc(\pi_1\times\pi_2)$ where $\pi_1$ is a basic $\sigma$-saturated representation.
Correspondingly, $C=C_1*C_2$ with $C_1$ a subcomponent of $C_\sigma$, in contradiction to the assumption that $C$
is $\sigma$-reduced.
\end{proof}

\section{Proof of Theorem \ref{thm: cor}} \label{sec: endofproof}

Let $\sigma$ be a basic representation and $C$ an irreducible component with $\sigma$-decomposition $C=C_1*C_2$.
It follows from \cite[Corollary 8.7]{2103.12027} that if $C$ is rigid, then $C_2$ is rigid.
The converse is not true in general. However, we have the following. (As usual, we denote by $C_\sigma$ the irreducible component
corresponding to $\sigma$.)
\begin{proposition} \label{prop: indelta}
Let $\sigma$ be a basic representation and $C$ an irreducible component with $\sigma$-decomposition $C=C_1*C_2$.
Assume that
\begin{enumerate}
\item $C$ commutes with every subcomponent of $C_\sigma$.
\item $C_2$ is rigid.
\item $\pi(C_2)$ is $\square$-irreducible.
\item $\pi(C_2*D)=\soc(\pi(C_2)\times\pi(D))$ for every $D\in\Irrcomp$.
\end{enumerate}
Then,
\begin{enumerate}
\item $C$ is rigid.
\item $\pi(C)$ is $\square$-irreducible.
\item $\pi(C*D)=\soc(\pi(C)\times\pi(D))$ for every $D\in\Irrcomp$.
\end{enumerate}
\end{proposition}

\begin{proof}
The first part follows from \cite[Corollary 8.2]{2103.12027} since
$C_1$ and $C_2$ are rigid and $C$ commutes with $C_1$.
The second part follows from Lemma \ref{lem: sqr123} since $\pi(C)\times\pi(C_1)$ is irreducible.

It remains to show that
\[
\pi(C*D)=\soc(\pi(C)\times\pi(D))
\]
for any irreducible component $D$.

Assume first that $D$ is $\sigma$-reduced.
Then, since $C_1$ is $\sigma$-saturated, we have $\hom_\Pi(D,C_1)=0$ and hence by \cite[Corollary 7.4]{2103.12027}
\[
C*D=(C_1*C_2)*D=C_1*(C_2*D).
\]
On the other hand, $\pi(C_2)\times\pi(D)$ is $\sigma$-reduced by \eqref{eq: prodred} and SI since $\pi(C_2)$ is $\square$-irreducible.
Hence, $\pi(C_1)\times\pi(C_2)\times\pi(D)$ is SI by Lemma \ref{lem: prodSI}. Therefore,
\begin{multline*}
\soc(\pi(C)\times\pi(D))=\soc(\pi(C_1)\times\pi(C_2)\times\pi(D))\\
=\soc(\pi(C_1)\times\soc(\pi(C_2)\times\pi(D))).
\end{multline*}
Since $C_1$ is good, by the assumption on $C_2$ we get
\[
\soc(\pi(C)\times\pi(D))=\pi(C_1*(C_2*D))=\pi(C*D).
\]

In the general case, let $D=D_1*D_2$ be the $\sigma$-decomposition of $D$.
By assumption on $C$ and Lemma \ref{lem: satcomp}, $C$ and $D_1$ commute.
Hence, by \cite[Corollary 7.4]{2103.12027}
\[
C*D=C*(D_1*D_2)=(C*D_1)*D_2=(D_1*C)*D_2=D_1*(C*D_2).
\]
Moreover, $\pi(C)\times\pi(D_1)$ is irreducible, hence $\square$-irreducible
since $\pi(C)$ and $\pi(D_1)$ are both $\square$-irreducible. Hence
$\pi(C)\times\pi(D_1)\times\pi(D_2)$ is SI, and therefore
\begin{multline*}
\soc(\pi(C)\times\pi(D))=\soc(\pi(C)\times\pi(D_1)\times\pi(D_2))\\=\soc(\pi(D_1)\times\soc(\pi(C)\times\pi(D_2))).
\end{multline*}
We already proved that $\soc(\pi(C)\times\pi(D_2))=\pi(C*D_2)$. Therefore,
\[
\soc(\pi(C)\times\pi(D))=\pi(D_1*(C*D_2))=\pi(C*D).
\]
The proposition follows.
\end{proof}

Recall that by Remark \ref{rem: symm}, in order to show Theorem \ref{thm: cor},
it is enough to show that if $C=\prm_Q(\m)$ where $\m$ is (regular) balanced, then
\[
\pi(C*D)=\soc(\pi(C)\times\pi(D))
\]
for every $D\in\Irrcomp$. We prove this, together with the rigidity of $C$ and the $\square$-irreducibility of $Z(\m)$
by induction on the size of $\m$.
The induction step follows from Proposition \ref{prop: indelta} using the following combinatorial lemma
which is similar to \cite[Lemma 7.4]{MR3866895}, and proved along the same lines.

\begin{lemma} \label{lem: combreg}
Suppose $C=\prm_Q(\m)$ where $\m\ne0$ is balanced.
Then, there exists a basic irreducible component $C_\sigma$ and a $\sigma$-reduced irreducible
component $C'=\prm_Q(\m')$ such that
\begin{enumerate}
\item $C=C_\sigma*C'$.
\item $\m'$ is balanced.
\item $C$ commutes with every subcomponent of $C_\sigma$.
\end{enumerate}
\end{lemma}

\begin{proof}
Write $\m=\Delta_1+\dots+\Delta_k$ where $b(\Delta_1)>\dots>b(\Delta_k)$.
Let $1\le m\le k$ be the largest index such that $e(\Delta_1)>\dots>e(\Delta_m)$.
Let $1\le l\le m$ be the smallest index such that $\Delta_{i+1}\prec\Delta_i$ for all $l\le i<m$.
(In other words, $l$ is the largest index between $2$ and $m$ such that $b(\Delta_{l-1})>e(\Delta_l)+1$
if such an index exists; otherwise $l=1$.)

The proof splits into two cases according to whether there exist indices $i$ and $j$ such that $l\le i<m<j\le k$ and $e(\Delta_{i+1})<e(\Delta_j)<e(\Delta_i)$.
Assume first that such indices do not exist.
We will show that the required conditions are satisfied for $\sigma=Z(\Delta)$ and $\m'=\m-\Delta$ where we set $\Delta=\Delta_l$.
Clearly, $\m'$ is balanced, since it is a submultisegment of $\m$.
It is also clear from the definition of $l$ that $\Delta\not\prec\Delta_i$ for all $i$. Hence, $C=C_\sigma*C'$.
Moreover, $C'$ is $\sigma$-reduced since $e(\Delta_i)\ne e(\Delta)$ for all $i\ne l$.
Suppose that $\sigma'=Z(\Delta')$ where $\Delta'\subset\Delta$ and $e(\Delta')=e(\Delta)$.
In order to show that $C$ commutes with $C_{\sigma'}$
we use Remark \ref{rem: matching} and its notation. The condition $\mtch(\m,\Delta')$ is trivial since $U_{\m;\Delta'}=\emptyset$.
In order to show $\mtch(\Delta',\m)$ it suffices to show that
if $X=\{j\mid\Delta_j\prec\Delta'\}$ and $Y=\{i\mid\Delta_i\prec\rshft{\Delta'}\}$, then the function
\[
f:X\rightarrow Y,\ \ f(j)=\max\{i\in Y\mid \Delta_j\prec\Delta_i\}
\]
is injective. Note that $f$ is well-defined and $l\le f(j)<j$ for all $j$.

Assume on the contrary that $f(j)=f(j')$ for some $j,j'\in X$ with $j'<j$.
Clearly, $\Delta_j\not\prec\Delta_{j'}$, otherwise $f(j)\ge j'>f(j')$.
On the other hand $b(\Delta_{j'})<b(\Delta')\le e(\Delta_j)+1$.
Since also $b(\Delta_j)<b(\Delta_{j'})$, we necessarily have $e(\Delta_j)>e(\Delta_{j'})$. In particular, $j>m$.
Let $i$ be the largest index between $l$ and $m$ such that $\Delta_{j'}\prec\Delta_i$ and $\Delta_j\prec\Delta_i$.
If $i=m$, then $j'>m+1$ (since $e(\Delta_{m+1})>e(\Delta_m)$) and $\Delta_m+\Delta_{m+1}+\Delta_{j'}+\Delta_j$ is a submultisegment of type $3412$,
which is a contradiction to the assumption on $\m$. Thus, $i<m$ and therefore $e(\Delta_{i+1})<e(\Delta_i)$.
By the maximality of $i$, we necessarily have $e(\Delta_{i+1})<e(\Delta_j)$,
in contradiction to our assumption on $\m$. 

It remains to consider the case where there exist indices $r$ and $s$ such that $l\le r<m<s\le k$ and $e(\Delta_{r+1})<e(\Delta_s)<e(\Delta_r)$.
Let $1\le l'\le m$ be the smallest index such that $b(\Delta_i)=b(\Delta_{i+1})+1$ for all $l'\le i<m$.
(Clearly, $l'\ge l$.)
Let $\sigma=L([b(\Delta_m),b(\Delta_{l'})])$ and $\m'=\Delta'_1+\dots+\Delta'_k$ where
\[
\Delta'_i=\begin{cases} ^-\Delta_i&i=l',\dots,m\\\Delta_i&\text{otherwise.}\end{cases}
\]
(We omit indices for which $\Delta'_i$ is empty.)
Since $b(\Delta'_i)\ne b(\Delta_m)$ for all $i$, $C'$ is $\sigma$-reduced.
Also, by the choice of $l'$, $\m'$ is regular and it is balanced since any submultisegment of $\m'$ of type $4231$ or of type $3412$ would give
(upon replacing $\Delta'_i$ by $\Delta_i$) a submultisegment of $\m$ of the same type.

Next, we show that $C=C_\sigma*C'$.
Let
\begin{gather*}
\m_1=\m_1'=\Delta_1+\dots+\Delta_{l'-1},\\
\m_2=\Delta_{l'}+\dots+\Delta_m,\ \ \m'_2=\Delta'_{l'}+\dots+\Delta'_m, \\
\m_3=\m_3'=\Delta_{m+1}+\dots+\Delta_k,\\
\m_4=\m_2+\m_3,\ \m'_4=\m'_2+\m'_3,\\
C_i=\prm_Q(\m_i), C'_i=\prm_Q(\m'_i),\ i=1,2,3,4.
\end{gather*}
By Lemma \ref{lem: notprec1} we have
\[
C_2*C_3=C_4,\ \ C'_2*C'_3=C'_4,\ \ C_1*C_4=C,\ \ C'_1*C'_4=C'.
\]
Moreover, by the choice of $l'$, $C_\sigma$ and $C'_1$ commute. Therefore,
\[
C_\sigma*C'=C_\sigma*(C'_1*C'_4)=(C_\sigma*C'_1)*C'_4=(C'_1*C_\sigma)*C'_4=
C'_1*(C_\sigma*C'_4).
\]
Also, $\hom_\Pi(C'_3,C_\sigma)=0$ and therefore,
\[
C_\sigma*C'_4=C_\sigma*(C'_2*C'_3)=(C_\sigma*C'_2)*C'_3.
\]
Note that $C_\sigma*C'_2=C_2$ (see Example \ref{exam: twolads}). We therefore get
\[
C_\sigma*C'_4=C_2*C'_3=C_2*C_3=C_4.
\]
Hence,
\[
C_\sigma*C'=C'_1*C_4=C_1*C_4=C.
\]

It remains to show that $C$ commutes with every subcomponent $\prm_Q(\n')$ of $C_\sigma$.
Note that $\n'=[b(\Delta_m),b(\Delta_m)]+\dots+[s,s]$ for some $s\in [b(\Delta_m),b(\Delta_{l'})]$.
We use Remark \ref{rem: matching}. The condition $\mtch(\m,\n')$ follows from the fact that
$\Delta_{i+1}\prec\Delta_i$ and $b(\Delta_{i+1})=b(\Delta_i)-1$ whenever
$b(\Delta_i)\in [b(\Delta_m)+1,b(\Delta_{l'})+1]$.
In order to show $\mtch(\n',\m)$, it suffices to check that $e(\Delta_i)\ne b(\Delta_m)-1$ for all $i$.

Assume on the contrary that there exists $i$ such that $e(\Delta_i)=b(\Delta_m)-1$.
Clearly, $i\ne s$ since $e(\Delta_s)>e(\Delta_m)>e(\Delta_i)$.
Necessarily $i<s$, otherwise $\Delta_r+\dots+\Delta_m+\Delta_s+\Delta_i$ is a submultisegment of type $4231$.
By definition of $m$, $e(\Delta_{m+1})>e(\Delta_m)$. Suppose that $e(\Delta_{m+1})<e(\Delta_l)$.
Let $l\le t<m$ be the index such that $e(\Delta_{t+1})<e(\Delta_{m+1})<e(\Delta_t)$.
Then, $\Delta_t+\dots+\Delta_{m+1}+\Delta_i$ is a submultisegment of type $4231$, in contradiction to
the assumption on $\m$. Otherwise, $e(\Delta_{m+1})>e(\Delta_l)$ and hence, $\Delta_r+\dots+\Delta_{m+1}+\Delta_i+\Delta_s$
is a submultisegment of type $3412$. Once again we get a contradiction.

The proof of the lemma is complete.
\end{proof}

As indicated before, Theorem \ref{thm: cor} now follows from Remark \ref{rem: symm}, Proposition \ref{prop: indelta}
and Lemma \ref{lem: combreg} by induction on the size of $\m$.

\begin{remark} \label{rem: recipebalanced}
The proofs of Proposition \ref{prop: indelta} and Lemma \ref{lem: combreg}, together with Remark \ref{rem: recipe1},
give a combinatorial recipe for the computation of $\m*\n$ in the case where $\m$ or $\n$ is balanced.
\end{remark}

\begin{remark}
Proposition \ref{prop: indelta} can be used to provide additional examples of rigid, good irreducible components.
It would be interesting to see what are the limits of this method.
At any rate, there are examples of rigid irreducible components $C\ne0$ that do not commute with
any basic irreducible component $C_\sigma$ for which $C=C_\sigma*C'$ for some $C'$.
\end{remark}

\begin{remark}
It would be interesting to extend the concepts and results of \S\ref{sec: defbasic}--\S\ref{sec: basicomp}
to the more general framework of \cite{MR3270589}.
\end{remark}

\section{Odds and ends} \label{sec: odds}

We end the paper with several questions and speculations.

\begin{question}
Given $C_1,C_2\in\Irrcomp$, what is the representation-theoretic counterpart of the condition $\hom_\Pi(C_1,C_2)=0$ ?
\end{question}

Recall the condition $\sprt$ introduced in \S\ref{sec: fork}. One may wonder whether the conditions
$\hom_\Pi(C_2,C_1)=0$ and $\pi(C_1)\sprt\pi(C_2)$ are equivalent.
At this stage we are unable to prove or disprove either direction.

Let $C_i\in\Irrcomp(\bfd_i)$, $i=1,\dots,k$ and let $\bfd=\bfd_1+\dots+\bfd_k$.
For any $C\in\Irrcomp(\bfd)$ we would like to give a criterion for the condition
\begin{equation} \label{eq: occurs}
\pi(C)\text{ occurs as a subquotient of }\pi(C_1)\times\dots\times\pi(C_k).
\end{equation}

Fix graded vector space $V^i$ of graded dimension $\bfd_i$, $i=1,\dots,k$ and let $V=V^1\oplus\dots\oplus V^k$.
We identify $\Lambda(V^1)\times\dots\times\Lambda(V^k)$ with a closed subvariety of
$\Lambda(V)$ by the embedding
\[
(x_1,\dots,x_k)\mapsto x_1\oplus\dots\oplus x_k.
\]
Thus, by definition,
\[
C_1\oplus\dots\oplus C_k\text{ is the closure of }G_V\cdot (C_1\times\dots\times C_k).
\]
Let
\[
p_i:\Lambda(V^1)\times\dots\times\Lambda(V^k)\rightarrow\Lambda(V^i)
\]
be the projection onto the $i$-th factor.

We make the following conjecture.\footnote{We thank Avraham Aizenbud and Jan Schr\"oer for discussions related to this conjecture.}

\begin{conjecture} \label{conj: genprod}
For any $C\in\Irrcomp(V)$, \eqref{eq: occurs} holds
if and only if there exists an irreducible component $D$ of $C\cap(C_1\times\dots\times C_k)$ such that
for all $i=1,\dots,k$ the restriction of $p_i$ is a dominant map from $D$ to $C_i$.

In particular,
\begin{enumerate}
\item If all but at most one of the $C_i$'s is rigid, then
\[
\JH(\pi(C_1)\times\dots\times\pi(C_k))=\{\pi(C)\mid C\supset C_1\oplus\dots\oplus C_k\}.
\]
\item If all $C_i$'s are rigid, then
\[
\JH(\pi(C_1)\times\dots\times\pi(C_k))=\{\pi(C)\mid x_1\oplus\dots\oplus x_k\in C\}.
\]
where $x_i$ is a rigid element of $C_i$, $i=1,\dots,k$.
\end{enumerate}
\end{conjecture}

At the moment, we have very little evidence towards this conjecture.
However, we can make a few consistency checks and remarks.
\begin{enumerate}
\item (Cf.\ \cite[(9.7)]{2103.12027} and Remark \ref{rem: simsub}).
Suppose that $C_i=\prm_Q(\m_i)$, $i=1,\dots,k$. Then,
\[
\prm_Q(\m_1+\dots+\m_k)\supset C_1\oplus\dots\oplus C_k
\]
while
\[
Z(\m_1+\dots+\m_k)\text{ occurs as a subquotient of }Z(\m_1)\times\dots\times Z(\m_k)
\]
(in fact, with multiplicity one).

\item Let $\m=\Delta_1+\dots+\Delta_k$ and take $C_i=\prm_Q(\Delta_i)$.
Then, it is clear that $\prm_Q(\m')\supset C_1\oplus\dots\oplus C_k$ if and only if $M_Q(\m')$
lies in the closure of the $G_V$-orbit of $M_Q(\m)$. On the other hand, by \cite[\S7]{MR584084},
this condition also characterizes the occurrence of $Z(\m')$ as a subquotient in the standard module $\std(\m)$.
Thus, Conjecture \ref{conj: genprod} is true in this case.
The same holds if we take $C_i=\prm_{Q^{\op}}(\Delta_i)$.

\item In the case $k=2$, by \cite[(3.7)]{2103.12027}, Conjecture \ref{conj: genprod} implies that
$\pi(C_1*C_2)$ is an irreducible subquotient of $\pi(C_1)\times\pi(C_2)$, which is a weak form
of Conjecture \ref{conj: subconj}.
In particular, as in Lemma \ref{lem: conjcons}, Conjecture \ref{conj: genprod} implies that if $C_1$
and $C_2$ do not strongly commute, then $\pi(C_1)\times\pi(C_2)$ is reducible.
Also, Conjecture \ref{conj: genprod} implies Conjecture \ref{conj: GS2}.

\item Conjecture \ref{conj: genprod} implies that if $C$ is rigid, then $\pi(C)$ is $\square$-irreducible.
By the above, Conjecture \ref{conj: genprod} also implies that if $C$ does not strongly commute with itself, then $\pi(C)$ is not $\square$-irreducible.
We expect that as in Example \ref{ex: nonrigid}, if $C$ strongly commutes with itself but is not rigid,
then there exists an open, nonempty subset $C'\subset C$
such that for every $x\in C'$ there exists  a self-extension of $x$ that is not contained in $C*C=C\oplus C$.
This would imply that assuming Conjecture \ref{conj: genprod}, if $C$ is non-rigid, then $\pi(C)$ is not $\square$-irreducible.
However, as things stand, we do not know whether Conjecture \ref{conj: genprod} implies Conjecture \ref{conj: GSR}.

\item In general, by \cite{MR1944812}, $C_1\oplus\dots\oplus C_k$ is an irreducible component if and only if
$C_i$ strongly commutes with $C_j$ for all $i\ne j$. If this is the case, and in addition all but at most one of the $C_i$'s is rigid, then
Conjecture \ref{conj: genprod} asserts that $\pi(C_1)\times\dots\times\pi(C_k)$ is irreducible.
This is consistent with Conjectures \ref{conj: GS2} and \ref{conj: GSR}, in view of Corollary \ref{cor: allprod}.
\end{enumerate}

Part of the subtlety with Conjecture \ref{conj: genprod}, already in the case where all $C_i$'s are rigid,
is to get a handle on the condition
\[
C\supset C_1\oplus\dots\oplus C_k.
\]
This condition clearly implies that for every irreducible component $D$
\begin{gather*}
\hom_\Pi(D,C)\le\hom_\Pi(D,C_1)+\dots+\hom_\Pi(D,C_k),\\
\hom_\Pi(C,D)\le\hom_\Pi(C_1,D)+\dots+\hom_\Pi(C_k,D).
\end{gather*}
We do not know whether the converse implication also holds, let alone whether it suffices to consider
only certain $D$'s, which would make this condition feasible to check.
(A related problem was studied by Riedtmann, Zwara and others, see e.g., \cite{MR868301, MR1757882}.)

\begin{example}
Let $\Delta_i=[i,i+r]$, $i=1,\dots,r$ and $k\le r$. For any $i=1,\dots,k$ consider
\[
C_i=\prm_Q(\m_i)\text{ where }\m_i=\sum_{1\le j\le r\mid j\equiv i\bmod k}\Delta_j.
\]
It is known that any irreducible subquotient of $Z(\m_1)\times\dots\times Z(\m_k)$ is of the form $Z(\m)$
with $\m=\sum_{i=1}^r[w(i),i+r]$ where $w\in S_r$ is a permutation
which does not admit a decreasing subsequence of size $k+1$ \cite{MR4172670}.
It is also easy to see that if $C\supset C_1\oplus\dots\oplus C_k$, then $C=\prm_Q(\m)$ for $\m$ of this type.
Computer calculations affirms that conversely, at least up to $r=8$,
for such $\m$, $Z(\m)$ occurs in $Z(\m_1)\times\dots\times Z(\m_k)$.
We expect that this holds in general. (This is known for $k=2$, in which case the multiplicities are one \cite{MR4213657}.)
It would be already interesting to see that this expectation is consistent with Conjecture \ref{conj: genprod}.
\end{example}

\def\cprime{$'$}
\providecommand{\bysame}{\leavevmode\hbox to3em{\hrulefill}\thinspace}
\providecommand{\MR}{\relax\ifhmode\unskip\space\fi MR }
\providecommand{\MRhref}[2]{%
  \href{http://www.ams.org/mathscinet-getitem?mr=#1}{#2}
}
\providecommand{\href}[2]{#2}


\begin{thebibliography}{KKKO18}

\bibitem[AL]{2103.12027}
Avraham Aizenbud and Erez Lapid, \emph{A binary operation on irreducible
  components of {L}usztig's nilpotent varieties {I}: definition and
  properties}, arXiv:2103.12027.

\bibitem[BKT14]{MR3270589}
Pierre Baumann, Joel Kamnitzer, and Peter Tingley, \emph{Affine
  {M}irkovi\'{c}-{V}ilonen polytopes}, Publ. Math. Inst. Hautes \'{E}tudes Sci.
  \textbf{120} (2014), 113--205. \MR{3270589}

\bibitem[BZ76]{MR0425030}
I.~N. Bern\v{s}te\u{\i}n and A.~V. Zelevinski\u{\i}, \emph{Representations of
  the group {$GL(n,F),$} where {$F$} is a local non-{A}rchimedean field},
  Uspehi Mat. Nauk \textbf{31} (1976), no.~3(189), 5--70. \MR{0425030}

\bibitem[BZ77]{MR579172}
\bysame, \emph{Induced representations of
  reductive {${\germ p}$}-adic groups. {I}}, Ann. Sci. \'{E}cole Norm. Sup. (4)
  \textbf{10} (1977), no.~4, 441--472. \MR{579172}

\bibitem[CBS02]{MR1944812}
William Crawley-Boevey and Jan Schr\"{o}er, \emph{Irreducible components of
  varieties of modules}, J. Reine Angew. Math. \textbf{553} (2002), 201--220.
  \MR{1944812}

\bibitem[Deo85]{MR788771}
Vinay~V. Deodhar, \emph{Local {P}oincar\'e duality and nonsingularity of
  {S}chubert varieties}, Comm. Algebra \textbf{13} (1985), no.~6, 1379--1388.
  \MR{788771 (86i:14015)}

\bibitem[EGNO15]{MR3242743}
Pavel Etingof, Shlomo Gelaki, Dmitri Nikshych, and Victor Ostrik, \emph{Tensor
  categories}, Mathematical Surveys and Monographs, vol. 205, American
  Mathematical Society, Providence, RI, 2015. \MR{3242743}

\bibitem[GLS06]{MR2242628}
Christof Gei\ss, Bernard Leclerc, and Jan Schr\"{o}er, \emph{Rigid modules over
  preprojective algebras}, Invent. Math. \textbf{165} (2006), no.~3, 589--632.
  \MR{2242628}

\bibitem[GM21]{MR4272560}
Maxim Gurevich and Alberto M\'{\i}nguez, \emph{Cyclic representations of
  general linear {$p$}-adic groups}, J. Algebra \textbf{585} (2021), 25--35.
  \MR{4272560}

\bibitem[GS05]{MR2115084}
Christof Geiss and Jan Schr\"{o}er, \emph{Extension-orthogonal components of
  preprojective varieties}, Trans. Amer. Math. Soc. \textbf{357} (2005), no.~5,
  1953--1962. \MR{2115084}

\bibitem[Gur20]{MR4172670}
Maxim Gurevich, \emph{Decomposition rules for the ring of representations of
  non-{A}rchimedean {$GL_n$}}, Int. Math. Res. Not. IMRN (2020), no.~20,
  6815--6855. \MR{4172670}

\bibitem[Gur21a]{MaxChari2021}
\bysame, \emph{On the {H}ecke-algebraic approach for general linear groups over
  a p-adic field}, Interactions of Quantum Affine Algebras with Cluster
  Algebras, Current Algebras and Categorification, Progress in Mathematics,
  vol. 337, Springer, 2021, [In Honor of Vyjayanthi Chari on the Occasion of
  Her 60th Birthday].

\bibitem[Gur21b]{MR4213657}
\bysame, \emph{Quantum invariants for decomposition problems in type {$A$}
  rings of representations}, J. Combin. Theory Ser. A \textbf{180} (2021),
  Paper No. 105431, 46. \MR{4213657}

\bibitem[Hen07]{MR2320806}
Anthony Henderson, \emph{Nilpotent orbits of linear and cyclic quivers and
  {K}azhdan-{L}usztig polynomials of type {A}}, Represent. Theory \textbf{11}
  (2007), 95--121 (electronic). \MR{2320806}

\bibitem[HL10]{MR2682185}
David Hernandez and Bernard Leclerc, \emph{Cluster algebras and quantum affine
  algebras}, Duke Math. J. \textbf{154} (2010), no.~2, 265--341. \MR{2682185}

\bibitem[HL21]{1902.01432}
\bysame, \emph{Quantum affine algebras and cluster algebras}, Interactions of
  Quantum Affine Algebras with Cluster Algebras, Current Algebras and
  Categorification, Progress in Mathematics, vol. 337, Birkh\"auser/Springer,
  Basel, 2021, [In Honor of Vyjayanthi Chari on the Occasion of Her 60th
  Birthday].

\bibitem[Kas18]{MR3966729}
Masaki Kashiwara, \emph{Crystal bases and categorifications---{C}hern {M}edal
  lecture}, Proceedings of the {I}nternational {C}ongress of
  {M}athematicians---{R}io de {J}aneiro 2018. {V}ol. {I}. {P}lenary lectures,
  World Sci. Publ., Hackensack, NJ, 2018, pp.~249--258. \MR{3966729}

\bibitem[KKKO15]{MR3314831}
Seok-Jin Kang, Masaki Kashiwara, Myungho Kim, and Se-jin Oh, \emph{Simplicity
  of heads and socles of tensor products}, Compos. Math. \textbf{151} (2015),
  no.~2, 377--396. \MR{3314831}

\bibitem[KKKO18]{MR3758148}
\bysame, \emph{Monoidal categorification of cluster algebras}, J. Amer. Math.
  Soc. \textbf{31} (2018), no.~2, 349--426. \MR{3758148}

\bibitem[KL12]{MR2996769}
Arno Kret and Erez Lapid, \emph{Jacquet modules of ladder representations}, C.
  R. Math. Acad. Sci. Paris \textbf{350} (2012), no.~21-22, 937--940.
  \MR{2996769}

\bibitem[KS97]{MR1458969}
Masaki Kashiwara and Yoshihisa Saito, \emph{Geometric construction of crystal
  bases}, Duke Math. J. \textbf{89} (1997), no.~1, 9--36. \MR{1458969
  (99e:17025)}

\bibitem[KZ96]{MR1371654}
Harold Knight and Andrei Zelevinsky, \emph{Representations of quivers of type
  {$A$} and the multisegment duality}, Adv. Math. \textbf{117} (1996), no.~2,
  273--293. \MR{1371654 (97e:16029)}

\bibitem[Lap21]{1911.04270}
Erez Lapid, \emph{Explicit decomposition of certain induced representations of
  the general linear group}, Relative trace formulas, Simons Symp., Springer,
  Cham, 2021, pp.~321--327.

\bibitem[Lec03]{MR1959765}
B.~Leclerc, \emph{Imaginary vectors in the dual canonical basis of {$U_q(\germ
  n)$}}, Transform. Groups \textbf{8} (2003), no.~1, 95--104. \MR{1959765}

\bibitem[LM14]{MR3163355}
Erez Lapid and Alberto M{\'{\i}}nguez, \emph{On a determinantal formula of
  {T}adi\'c}, Amer. J. Math. \textbf{136} (2014), no.~1, 111--142. \MR{3163355}

\bibitem[LM16]{MR3573961}
\bysame, \emph{On parabolic induction on inner forms of the general linear
  group over a non-archimedean local field}, Selecta Math. (N.S.) \textbf{22}
  (2016), no.~4, 2347--2400. \MR{3573961}

\bibitem[LM18]{MR3866895}
\bysame, \emph{Geometric conditions for
  {$\square$}-irreducibility of certain representations of the general linear
  group over a non-archimedean local field}, Adv. Math. \textbf{339} (2018),
  113--190. \MR{3866895}

\bibitem[LM20]{MR4169050}
\bysame, \emph{Conjectures and results about
  parabolic induction of representations of {${\rm GL}_n(F)$}}, Invent. Math.
  \textbf{222} (2020), no.~3, 695--747. \MR{4169050}

\bibitem[LNT03]{MR1985725}
Bernard Leclerc, Maxim Nazarov, and Jean-Yves Thibon, \emph{Induced
  representations of affine {H}ecke algebras and canonical bases of quantum
  groups}, Studies in memory of {I}ssai {S}chur ({C}hevaleret/{R}ehovot, 2000),
  Progr. Math., vol. 210, Birkh\"auser Boston, Boston, MA, 2003, pp.~115--153.
  \MR{1985725}

\bibitem[LS90]{MR1051089}
V.~Lakshmibai and B.~Sandhya, \emph{Criterion for smoothness of {S}chubert
  varieties in {${\rm Sl}(n)/B$}}, Proc. Indian Acad. Sci. Math. Sci.
  \textbf{100} (1990), no.~1, 45--52. \MR{1051089}

\bibitem[Lus90]{MR1035415}
G.~Lusztig, \emph{Canonical bases arising from quantized enveloping algebras},
  J. Amer. Math. Soc. \textbf{3} (1990), no.~2, 447--498. \MR{1035415}

\bibitem[Lus91]{MR1088333}
\bysame, \emph{Quivers, perverse sheaves, and quantized enveloping algebras},
  J. Amer. Math. Soc. \textbf{4} (1991), no.~2, 365--421. \MR{1088333}

\bibitem[MgW86]{MR863522}
C.~M\oe~glin and J.-L. Waldspurger, \emph{Sur l'involution de {Z}elevinski}, J.
  Reine Angew. Math. \textbf{372} (1986), 136--177. \MR{863522}

\bibitem[Rie86]{MR868301}
Christine Riedtmann, \emph{Degenerations for representations of quivers with
  relations}, Ann. Sci. \'{E}cole Norm. Sup. (4) \textbf{19} (1986), no.~2,
  275--301. \MR{868301}

\bibitem[Rin99]{MR1734897}
Claus~Michael Ringel, \emph{The multisegment duality and the preprojective
  algebras of type {$A$}}, AMA Algebra Montp. Announc. (1999), Paper 2, 6.
  \MR{1734897}

\bibitem[Sch80]{MR594695}
J.~T. Schwartz, \emph{Fast probabilistic algorithms for verification of
  polynomial identities}, J. Assoc. Comput. Mach. \textbf{27} (1980), no.~4,
  701--717. \MR{594695}

\bibitem[Zel80]{MR584084}
A.~V. Zelevinsky, \emph{Induced representations of reductive {${\germ p}$}-adic
  groups. {II}. {O}n irreducible representations of {${\rm GL}(n)$}}, Ann. Sci.
  \'Ecole Norm. Sup. (4) \textbf{13} (1980), no.~2, 165--210. \MR{584084
  (83g:22012)}

\bibitem[Zwa00]{MR1757882}
Grzegorz Zwara, \emph{Degenerations of finite-dimensional modules are given by
  extensions}, Compositio Math. \textbf{121} (2000), no.~2, 205--218.
  \MR{1757882}

\end{thebibliography}
\end{document}